\documentclass[12pt]{compositio}

\usepackage{amssymb,epsfig,mathrsfs}
\usepackage{graphicx}
\usepackage{mathtools}

\def\Z{{   \mathbb Z }}
\def\Q{{\mathbb Q}}

\def\F{{\mathbb F}}

\def\SSA{{Schur $\sigma$-ancestor }}
\def\Meas{{\mathrm{Meas}}}
\def\Rbar{{\overline{R}}}
\def\A{{Y}}  
\def\see{{y}}  
\def\r{{h}}  
\def\psibar{{\Psi}}

\def\FF{{\mathcal{F}}}
\def\ch{{\mathrm{ch}}}
\def\Freq{{\mathrm{Freq}}}
\def\CC{{\mathcal{C}}}
\def\IPAD{{\mathrm{IPAD}}}
\def\magma{{\tt Magma}}
\def\pari{{\tt PARI/GP}}

\def\a{{\mathfrak a}}

\def\G{{G}}

\def\ab{{\mathrm{ab}}}
\def\ipad{{\mathrm{IPAD}}}

\def\Gal{\mathrm{Gal}}

\def\Aut{{\mathrm{Aut}}}

\newtheorem{theorem}{Theorem}[section]
\newtheorem{conjecture}[theorem]{Conjecture}
\newtheorem{thm-fr}{Th\'eor\`eme}[section]
\newtheorem{lem-fr}{Lemme}[section]
\newtheorem{cor-fr}{Corollaire}[section]
\newtheorem{lemma}[theorem]{Lemma}
\newtheorem{corollary}[theorem]{Corollary}

\newtheorem{prop-fr}[theorem]{Proposition}

\theoremstyle{remark}

\newtheorem{remark}[theorem]{Remark}
\newtheorem{rem-fr}[theorem]{Remarque}
\newtheorem{example}[theorem]{Example}

\newtheorem{definition}[theorem]{Definition}
\newtheorem{def-fr}[theorem]{D\'efinition}

\begin{document}

\title{Heuristics for $p$-class Towers of Imaginary Quadratic Fields}

\author{Nigel Boston}
\email{boston@math.wisc.edu}
\address{Department of Mathematics, University of Wisconsin - Madison,  
480 Lincoln Drive, Madison, WI 53706, USA}

\author{Michael R. Bush}
\email{bushm@wlu.edu}
\address{Department of Mathematics,
Washington and Lee University,
Lexington, Virginia 24450, USA}

\author{Farshid Hajir}
\email{hajir@math.umass.edu}
\address{Department of Mathematics \& Statistics, University of Massachusetts - Amherst, 710 N. Pleasant Street, Amherst MA 01003, USA}

\thanks{The research of the first author was supported by NSA Grant
  MSN115460.  
  }

\dedication{With An Appendix by Jonathan Blackhurst \\  ~ \\ Dedicated to Helmut Koch}

\begin{abstract}
Cohen and Lenstra have given a heuristic which, for a fixed odd prime
$p$, leads to many interesting predictions about the distribution of
$p$-class groups of imaginary quadratic fields.  We extend the
Cohen-Lenstra heuristic to a non-abelian setting by considering, for
each imaginary quadratic field $K$, the Galois group of the $p$-class
tower of $K$, i.e. $\G_K:=\Gal(K_\infty/K)$ where $K_\infty$ is the
maximal unramified $p$-extension of $K$.  By class field theory, the
maximal abelian quotient of $\G_K$ is isomorphic to the $p$-class
group of $K$.  For integers $c\geq 1$, we give a heuristic of
Cohen-Lenstra type for the maximal $p$-class $c$ quotient of $\G_K$ and
thereby give a conjectural formula for how frequently a given
$p$-group of $p$-class $c$ occurs in this manner.  In particular, we
predict that every finite Schur $\sigma$-group occurs as $G_K$ for
infinitely many fields $K$.  We present numerical data in support of
these conjectures.
\end{abstract}

\maketitle

\section{Introduction}

\subsection{Cohen-Lenstra Philosophy}

About 30 years ago, Cohen and Lenstra \cite{CL1,CL2} launched a
heuristic study of the distribution of class groups of number fields.
To focus the discussion, we restrict to a specialized setting.  Let
$p$ be an odd prime.  Among the numerous insights contained in the
work of Cohen and Lenstra, let us single out two and draw a
distinction between them: (1) There is a natural probability
distribution on the category of finite abelian $p$-groups for which
the measure of each $G$ is proportional to the reciprocal of the size
of $\Aut(G)$; and (2) the distribution of the $p$-part of class groups
of imaginary quadratic fields is the same as the Cohen-Lenstra
distribution of finite abelian $p$-groups.  The first statement, a
purely group-theoretical one, is quite accessible and Cohen and
Lenstra prove many beautiful facts about such distributions (not just
for abelian groups viewed as $\Z$-modules but also more generally for
modules over rings of integers of number fields) in the first part of
\cite{CL2}.  The second, and bolder, insight is much less accessible
at present but leads to striking number-theoretical predictions, only
a small number of which have been proven, but all of which agree with
extensive numerical data.  Note that (2) quantifies the notion that
the (rather elementary) necessary conditions for a group to occur as
the $p$-part of the class group of an imaginary quadratic field -
namely that it be a finite abelian $p$-group - should also be
sufficient.

In the decades since the publication of \cite{CL1,CL2}, the
application of (1) has been broadened to a number of other situations.
It should be noted, however, that there are many circumstances where
the weighting factor should also involve some power of the order of
$G$.  This includes recent investigations into variation of
Tate-Shafarevich groups, variation of $p$-class tower groups ($p$ odd)
for real quadratic fields (to be described in a subsequent paper by
the authors) and variation in presentations of $p$-groups as described
in \cite{B}. The case under consideration in the current paper,
however, does not involve these extra factors.

As regards the combination of (1) and (2), one can speak of a
``Cohen-Lenstra strategy,'' perhaps, as follows.  Suppose we have a
sequence $G_1, G_2, \ldots$ of $p$-groups (arising as invariants
attached to some kind of arithmetic objects, say).  One can hope to
identify a category $\CC$ of groups in which the sequence lies and to
assign to each $G$ in $\CC$ a positive real number $w(G)$ called
its weight; 
we would expect the size of $\Aut_\CC(G)$ (the set of
automorphisms of $G$ in the category $\CC$) to appear in the denominator
of $w(G)$.  We set $w_\CC=\sum_{G \in \CC} w(G)$ for the total weight of $\CC$, assumed to be a finite quantity.  Suppose we also define the
frequency with which any object $G$ of $\CC$ occurs in the sequence
$(G_n)_{n\geq 1}$ to be the limit
$$ \Freq(G)=\lim_{n \to \infty} \frac{\sum_{\nu=1}^n\ch_G(G_\nu)}{n}
$$ assuming this exists.  Here, $\ch_G(H)$ is the characteristic
function of $G$, taking the value $1$ if $H$ is isomorphic to $G$ (in
the category $\CC$) and $0$ otherwise. The Cohen-Lenstra philosophy
would then say that, assuming the sequence $(G_n)_{n\geq 1}$ is
sufficiently general and the category $\CC$ is correctly chosen, for
each $G \in \CC$ we would expect $\Freq(G)$ to equal the Cohen-Lenstra measure of $G$ in the category $\CC$, namely $w(G)/w_\CC$.  In such a situation, we can speak of the sequence
$(G_n)$ ``obeying a Cohen-Lenstra distribution for the category
$\CC$ equipped with the weight function $w$.''

As just some of the examples of applications of this philosophy we
cite Cohen-Martinet \cite{Cohen-Martinet}, Wittman \cite{Wittman}, and
Boston-Ellenberg \cite{BE}.  In the first two of these, the class
groups are in fact studied as modules over the group ring of the
Galois group.  In \cite{BE}, the groups under study are non-abelian,
and in fact the situation is slightly different because the base field
is fixed (to be $\Q$) and the ramifiying set varies; however the
essential Cohen-Lenstra idea appears to apply in that situation also.

\subsection{The Cohen-Lenstra heuristics for $p$-class groups}

For an algebraic number field $K$, we let $A_K$ be the $p$-Sylow
subgroup of its ideal class group.  If we allow $K$ to vary
over all imaginary quadratic fields, ordered according to increasing
absolute value of the discriminant $d_K$, the groups $A_K$ fluctuate
with no immediately apparent rhyme or reason.   When Cohen and
Lenstra investigated their \emph{cumulative} behavior, however, they found a
surprising pattern.  Namely, they asked what can be said about the
frequency with which a given group would occur as $A_K$ when the
fields $K$ are ordered by the magnitude of their discriminants.  Their
heuristic, described above, led them to many predictions, one of which
is the following conjecture.

\begin{conjecture}[(Cohen-Lenstra)]
\label{cohenlenstra}
Fix a finite abelian group $G=\Z/p^{r_1} \times \cdots
\times \Z/p^{r_g}$ of rank $g\geq 1$.  Among the imaginary quadratic fields
$K$ such that $A_{K}$ has rank $g$, ordered by discriminant, the
probability that $A_{K}$ is isomorphic to $G$ is
$$
\frac{1}{|\Aut(G)|} \, {p^{g^2}}  \prod_{k=1}^g (1-p^{-k})^2.
$$
\end{conjecture}

\begin{remark}
\label{abelian-weights}
We provide more detail on how the above conjecture
is related to the heuristic that groups should be weighted
according to the inverse of the size of an appropriate automorphism group.
For a finite abelian group $G$, if we define the Cohen-Lenstra weight of
$G$ to be simply $w(G)=1/|\Aut(G)|$, then 
it is a theorem of Hall \cite{Hall2} and, in a more general context, of
Cohen-Lenstra, that the total weight $w_p$ of all finite abelian $p$-groups
is given by
$$
w_p=\sum_{H} w(H) = \prod_{n\geq 1} (1-p^{-n})^{-1},
$$
where $\sum_H$ means the sum over the isomorphism classes of
finite abelian $p$-groups. 
By  \cite[p.~56]{CL2}, the probability that an abelian
$p$-group has generator rank $g$ is given by
$$
\frac{ \sum_{\{H:d(H)=g\}} w(H) }{ \sum_{H} w(H)} =
p^{-g^2} \prod_{n\geq 1}(1-p^{-n}) \prod_{k=1}^g (1-p^{-k})^{-2}.
$$ 
Thus, under the Cohen-Lenstra distribution, 
the probability that a randomly chosen abelian $p$-group of generator
rank $g$ is isomorphic to $G$ is given by 
$$
\frac{w(G)}{\sum_{\{H:d(H)=g\}} w(H)} = 
\frac{1}{|\Aut(G)|} \, p^{g^2}  \prod_{k=1}^g
(1-p^{-k})^2.
$$
Cohen-Lenstra's fundamental heuristic assumption (2) then yields
Conjecture \ref{cohenlenstra}.  
\end{remark}

\subsection{Heuristics for the distribution of $p$-class tower groups}

In this article, we continue to assume that $p$ is odd and consider a non-abelian extension of the
number-theoretical objects studied by Cohen and Lenstra, passing from
the $p$-part of the class group of a number field $K$
 to the pro-$p$ fundamental group of
the ring of integers of $K$, namely the Galois group of its maximal
everywhere unramified $p$-extension.  For brevity, henceforth we will
refer to these groups as ``$p$-class tower groups.''  The key fact, as
pointed out in Koch-Venkov \cite{KV}, is
that $p$-class tower groups of imaginary quadratic fields (and certain of their quotients) must
satisfy a ``Schur $\sigma$'' condition; the precise definitions are
given below.  

To each finite Schur $\sigma$-group, or more generally to 
each maximal $p$-class $c$ quotient of such a group, we attach a rational
number we call its measure; it is given by a count of how likely it is 
for a randomly chosen set of relations of a certain type to define
the given group.
Our main heuristic assumption then, is that for the sequence of $p$-class
tower groups of imaginary quadratic fields, ordered by discriminant,
or more generally for the sequence of maximal $p$-class $c$ quotients
of these $p$-class tower groups (where $c$ is any fixed whole number),
the frequency of any given group equals the measure of the group.

To describe our situation in more detail, we specify some notation to
be used throughout the paper.  For a pro-$p$ group $G$, we write
$$d(G)=\mathrm{dim}_{\Z/p\Z} H^1(G,\Z/p\Z), \qquad r(G)=\dim_{\Z/p\Z}
H^2(G,\Z/p\Z),$$ where the action of $G$ on $\Z/p\Z$ is trivial. These
invariants give, respectively, the generator rank and relation rank of $G$ as a
pro-$p$ group.  The Frattini subgroup of $G$, denoted $\Phi(G)$, is defined to
be the closure of $[G,G]G^p$.  The groups
$G^{\ab}=G/\overline{[G,G]}$ and $G/\Phi(G)$ are, respectively, the
maximal abelian quotient and maximal exponent-$p$ abelian quotient of
$G$.

To describe how we pass to a non-abelian generalization, recall that
if $K_1$ is the $p$-Hilbert class field of $K$, defined to be its
maximal abelian unramified $p$-extension, then there is a canonical
isomorphism $A_{K} \to \mathrm{Gal}(K_1/K)$ given by the Artin
reciprocity map.  Now, let us consider the field $K_\infty$ obtained
by taking the compositum of {\em all} finite unramified $p$-extensions
of $K$, not just the abelian ones.  We put
$\G_K=\mathrm{Gal}(K_\infty/K)$.  It is clear that the maximal abelian
quotient of $\G_K$ is isomorphic to $A_K$ and by Burnside $d(\G_K) =
d(A_K)$.

The central question we consider in this work is: For a fixed odd prime
$p$, as $K$ varies over all imaginary quadratic fields of ascending
absolute value of discriminant, what can one say about the variation
of the groups $\G_K$?

Naturally, this is a more difficult question than the variation of
class groups, even for venturing a guess.  Already, the group $\G_K$
is not always finite.  Indeed, in \cite{KV}, Koch and Venkov proved
that $\G_K$ is infinite if $d(\G_K)\geq 3$; they did so by taking into
account all the facts they had at their disposal about the group
$\G_K$.  Namely, $\G_K$ is a finitely generated pro-$p$ group with
finite abelianization and deficiency $0$ (meaning that
$r(\G_K)-d(\G_K)=0$) and admits an automorphism of order $2$ which
acts as inversion on its abelianization (complex conjugation is such
an automorphism, for example).  Since having zero deficiency is
equivalent to having trivial Schur multiplier in this context, Koch
and Venkov dubbed groups having this particular set of properties
``Schur $\sigma$-groups.''  In Section 2, we review some of the work
of Koch and Venkov on Schur $\sigma$-groups, and develop a method
via counting relations, of measuring how frequently a given group occurs as
the  maximal $p$-class $c$ quotient of Schur $\sigma$-groups.

Positing our main heuristic assumption that
a finite $p$-group $G$ arises as a $p$-class tower group over an
imaginary quadratic field with the same frequency as $G$ occurs as a
randomly chosen group among Schur $\sigma$-groups, in Section~3 we
arrive at the following Conjecture, which should be compared to the
Cohen-Lenstra Conjecture above.

\begin{conjecture}\label{our-conjecture}
Suppose $G$ is a finite $p$-group which is a Schur $\sigma$-group of
generator rank $g\geq 1$ or, more generally, suppose $c$ is a positive
integer and $G$ is the maximal $p$-class $c$ quotient of a Schur
$\sigma$-group.  Then, among the imaginary quadratic fields $K$ such
that $A_K$ has rank $g$, ordered by discriminant, the probability that
$\G_{K}$ (or in the fixed $p$-class case, the maximal $p$-class $c$
quotient of $G_K$) is isomorphic to $G$ is equal to 
\begin{eqnarray*}
\frac{1}{|\Aut_\sigma(G)|} \, {p^{g^2}} 
\prod_{k=1}^g (1 - p^{-k}) \prod_{k=1+g-\r}^g (1 - p^{-k}),
\end{eqnarray*}
where $\r$ is the difference between the $p$-multiplicator rank and
nuclear rank of $G$  and $\Aut_\sigma(G)$ is the centralizer
in $\Aut(G)$ of an
automorphism $\sigma$ of order $2$, acting as inversion on the abelianization of
$G$.  We note that $0 \leq \r \leq g$ and that $\r = g$ for Schur
$\sigma$-groups, in which case the above formula should be compared with 
that of Conjecture \ref{cohenlenstra}.
\end{conjecture}
\begin{remark}
It is important to realize that the formula here depends on an
additional group-theoretical conjecture, namely that all the Schur
$\sigma$-groups and their maximal $p$-class $c$ quotients satisfy a
kernel invariance property (KIP), see Definition~\ref{kip}.  This
condition is discussed in more detail in Section~\ref{subsection-kip}.
If it turns out that there exist maximal $p$-class $c$ quotients of
Schur $\sigma$-groups not satisfying KIP, then describing their
distribution may necessitate a more complicated weighting factor than
the one that appears in the formula of
Conjecture~\ref{our-conjecture}.  Regardless of the validity of the
formula, we demonstrate that our heuristics are compatible with the
statements made in the abelian setting. In particular, we show that
they imply Conjecture~\ref{cohenlenstra}.
\end{remark}

\begin{remark}
As in Remark \ref{abelian-weights}, the above conjecture
is related to a choice of weight function for finite Schur $\sigma$-groups.
Namely we introduce the weight
function $w'(G)=1/|\Aut_\sigma(G)|$ for finite Schur $\sigma$-groups
$G$.  Conjecture \ref{our-conjecture} then arises from the hypothesis
that, for a given finite Schur $\sigma$-group $G$, 
the density of $K$ for which $G_K$ is isomorphic to $G$ is equal to
$
{w'(G)}/{w_p}.
$

\end{remark}

\begin{remark}  
We do not make a direct prediction about how frequently a given
infinite Schur $\sigma$-group $G$ occurs as a $p$-class tower group,
but for every $c\geq 1$, the maximal $p$-class $c$ quotient of $G$ is
finite and the conjecture above applies to predict the density of
\emph{all} imaginary quadratic $K$ (including the ones where $G_K$
is infinite) for which the maximal $p$-class $c$ quotients
of $G_K$ and $G$ coincide.  If $G$ is a finite Schur $\sigma$-group,
then its generator rank $g$ is at most $2$. The first case of
Conjecture \ref{our-conjecture} predicts how frequently such a group
occurs as a $p$-class tower group for imaginary quadratic fields.
\end{remark}

\subsection{Numerical Evidence}

As theoretical evidence for their conjecture, Cohen and Lenstra were
able to show that a relatively cheap consequence of their heuristic
assumption, namely the prediction that the average value of
$3^{d_3(A_K)}$ (as $K$ ranges over all imaginary quadratic fields) is
$2$, is in fact a highly non-trivial theorem of Davenport and
Heilbronn \cite{DH}.  In more recent work, for example see \cite{bh},
Bhargava and his students have obtained deep refinements and
extensions of the Davenport-Heilbronn result, in particular verifying
further consequences of the Cohen-Lenstra and Cohen-Martinet
conjectures.

As regards numerical evidence, class groups of imaginary quadratic
fields can be computed via an efficient algorithm, and so the class
group computations available to Cohen and Lenstra were quite
extensive.  In \cite{CL2}, they derived many consequences of their
heuristic, every one of which matched and in some cases even
``explained'' the {\em observed} variation of the $p$-part of the
class group of imaginary quadratic fields.

In our non-abelian situation, we do not even know an algorithm for
determining whether $\G_K$ is finite, much less for computing it, so
the numerical investigation of our heuristic is bound to be more
tricky.  One of the first examples of a computation of $\G_K$ in the
literature appears in a 1934 article of Scholz and Taussky \cite{ST}:
for the field $\Q(\sqrt{-4027})$, with $p=3$, $A_K$ is elementary
abelian of rank 2 and the group $\G_K$ has size $243$ and is
isomorphic to the group denoted {\tt SmallGroup(243,5)} in the
terminology of the computer algebra software package \magma
(see~\cite{BEO}).  The method of Boston and Leedham-Green \cite{BL}
can be used for certain $K$ to produce a short list of candidates for
the isomorphism class of $\G_K$.  Unless $\G_K$ happens to be one of
a few small groups, it is difficult to identify the isomorphism type of 
$\G_K$ (see section \ref{section-IPAD} for more details, especially the 
proof of Theorem \ref{group_theory_predictions}).

In order to test our heuristic hypothesis, we considered what kind of
number-theoretical data (meaning about the groups $\G_K$) was within
reach, and settled on the following: we computed the class groups of
unramified extensions of $K$ of degree $1$ or $p$.  In terms of group
theory, this ``index $\leq p$ abelianization data'' or ``IPAD,''
describes the abelianization of $\G_K$ as well as those of its index
$p$ subgroups.  Though it is impractical at present to attempt the
complete computation of $\G_K$ for all fields $K$ within a given large
discriminant range, it was possible for us to compute the IPADs for
over $460,000$ fields with discriminant in the range $-10^8<d_K<0$ and
to compare the distribution of IPADs to the group-theoretical
prediction.

As a summary of the numerical evidence, the second to last column of
Table~\ref{last-table} in Section~5 lists the observed frequencies of
the most common IPADs with $p = 3$ and $g = 2$ over all imaginary
quadratic fields $K$ with $|d_K| < 10^8$. The last column then gives
the theoretical predictions based on our heuristic.  Given the
variability of the data and the general convergence trend toward the
predicted value, we believe that the data support our conjecture.

\subsection{Organization of the paper}

As in \cite{CL2}, we have separated the group theory, where we have
theorems, from the number theory, where we mostly make conjectures and
collect data. We develop some basic facts about Schur $\sigma$-groups
in Section 2 and introduce various measures in both the abelian and
non-abelian setting.  In Section 3, we give a precise formulation of
our conjecture describing the variation of Galois groups of $p$-towers
of imaginary quadratic fields. The distribution of IPADs of Schur
$\sigma$-groups is investigated in Section 4.  This investigation
yields a number of results which we prove using a mixture of theory
and computation, thanks to the powerful technique of organizing
$p$-groups via O'Brien's $p$-group generation algorithm \cite{O}.  The
number-theoretical data we have collected is summarized in Section 5;
see in particular, Tables 1 and 2.  The computations were carried out
using the symbolic algebra packages \magma~\cite{magma} and
\pari~\cite{pari}.  Finally, the appendix contains a proof, by
Blackhurst, of a group-theoretical fact needed in Section 2.

\section{Schur $\sigma$-groups}

\subsection{Preliminaries}\label{subsection-prelim}

Let $p$ be an odd prime.  

\begin{definition}
An automorphism of a finitely generated pro-$p$ group $G$ is called a
GI-automorphism (meaning ``generator-inverting") if it has order $2$
and acts as inversion on $G^{\ab}$.
\end{definition}

\begin{definition}
A finitely generated pro-$p$ group $G$ is called a Schur
$\sigma$-group of rank $g$ if it satisfies the following properties:
1) $d(G)=r(G)=g$; 2) $G^{\ab}$ is finite; 3) It has a GI-automorphism
$\sigma$.
\end{definition}
We now fix $g\geq 1$, and let $F$ denote the free pro-$p$ group on $g$
generators $x_1, \ldots, x_g$.  Let $\sigma$ be the automorphism of
$F$ induced by the assignment $\sigma(x_i)=x_i^{-1}$ for $i=1, \ldots,
g$. 
Koch and Venkov \cite{KV} showed that, given a GI-automorphism $\sigma$ on $G$, one can choose an epimorphism from $F$ to $G$ so that 
this automorphism is induced by the GI-automorphism $\sigma$ on $F$. In particular, this means that we can find generators for $G$ which lie in 
\[ X(G, \sigma) = \{s \in G \mid \sigma(s)=s^{-1}\}. \]
In addition, the relations of a Schur  $\sigma$-group can always be chosen to lie in
\[ X = X(\Phi(F), \sigma) = \{s \in \Phi(F) \mid \sigma(s)=s^{-1}\}. \]
Using refinements of the theorem of Golod and Shafarevich, Koch and Venkov proved that Schur $\sigma$-groups of rank $g \geq 3$
are always infinite.

In general, we will use the symbol $\sigma$ to denote both the specific automorphism of $F$ defined above and a general GI-automorphism on a group $G$ except when there is the potential for confusion.  Suppose $G$ is a pro-$p$ group and $\sigma$ is
a GI-automorphism of $G$.  As shown by Hall (section 1.3 of \cite{Hall1}, although
sometimes attributed to Burnside), the kernel from $\Aut(G) \rightarrow
\Aut(G/\Phi(G))$ is a pro-$p$ group and so by Schur-Zassenhaus
(e.g. Prop.~1.1 of \cite{GHR}), all lifts of order $2$ of the
inversion automorphism on $G/\Phi(G)$ are conjugate to each other.
It follows that the sets $X(G,\sigma)$ and $\A(G,\sigma)$ where 
\[ \A(G,\sigma) = \{ x \in G \mid \sigma(x) = x \}   \]
 are well-defined up to conjugacy and that their orders are independent of the choice of
GI-automorphism $\sigma$ and hence depend only on $G$. We will denote the order of $\A(G,\sigma)$ by $\see(G)$.
Also observe that $\A(G,\sigma) = \A(\Phi(G),\sigma)$. This follows since $p$ is odd and the automorphism induced by $\sigma$ on 
the elementary $p$-abelian quotient $G/\Phi(G)$ is inversion.

We now consider certain special finite quotients of
a finitely generated pro-$p$ group, namely their maximal quotients of a fixed $p$-class. To
define this, let $P_0(G) = G$ and, for $n \geq 0$, $P_{n+1}(G)$ denote
the (closed) subgroup generated by $[G,P_n(G)]$ and $P_n(G)^p$. The
groups $P_0(G) \geq P_1(G) \geq P_2(G) \geq \ldots$ form a descending
chain of characteristic subgroups of $G$ called the lower $p$-central series.  
Note that $P_1(G)$ is the
Frattini subgroup $\Phi(G)$.  The $p$-class $c$ of a finite $p$-group
$G$ is defined to be the smallest $n\geq 0$ for which $P_n(G) = \{1\}$.  If $N$ is a
normal subgroup of $G$, and $G/N$ has $p$-class $n$, then $P_n(G)\leq
N$.  Thus, if $G$ has $p$-class $c$, then for $n=0,\ldots, c$, the
maximal $p$-class $n$ quotient of $G$ is $G/P_n(G)$.  

Suppose $G$ has $p$-class $c$. A pro-$p$ group $H$ satisfying $H/P_c(H)
\cong G$ is called a {\em descendant} of $G$ and if, additionally, $H$
has $p$-class $c+1$, then $H$ is called a {\em child}, or 
{\em immediate descendant}, of $G$.  O'Brien \cite{O}
produced an algorithm that computes all children (and so
ultimately all descendants of any finite $p$-class) of a given $p$-group. 
It will be important
for us to give much consideration to the maximal $p$-class $n$ quotients
of Schur $\sigma$-groups so we make the following definition.

\begin{definition}
Let $G$ be a finite $p$-group of $p$-class $c$.  We say that $G$ is
a {\em \SSA} if it is the maximal $p$-class $c$ quotient of a Schur $\sigma$-group.
Note that this terminology has the slightly unorthodox meaning in that
every Schur $\sigma$-group is itself a \SSA.
\end{definition}

For the O'Brien $p$-group generation algorithm, two invariants of a
$p$-group $G$ play important roles namely its $p$-multiplicator rank
and its nuclear rank. We now recall their definitions and some of their
important properties. Suppose $G$ is a $p$-group with $d(G)=g$ and
presentation $1 \to R \to F \to G \to 1$; recall that $F$ is the free
pro-$p$ group on $g$ generators $x_1, \ldots, x_g$. The isomorphism
class of the objects we are about to define do not depend on the
choice of presentation. The \textit{$p$-covering group} $G^*$ of $G$
is $F/R^*$ where $R^*$ is the topological closure of $R^p[F,R]$.  The
\textit{$p$-multiplicator} of $G$ is defined to be the subgroup
$R/R^*$ of $G^*$, and the \textit{nucleus} of $G$ is $P_c(G^*)$ where
$c$ is the $p$-class of $G$. The nucleus is a subgroup of the
$p$-multiplicator.  We call the dimension of $R/R^*$ the
\textit{$p$-multiplicator rank}; the dimension of the subgroup
$P_c(G^*)$ is called the \textit{nuclear rank}.  If a
group has nuclear rank $0$, then it has no children and is called {\em terminal}.  

\begin{remark}
In  \cite{O}, the quantities introduced above are shown to be well-defined with respect to any choice of abstract presentation for a finite
$p$-group $G$ rather than for pro-$p$ presentations. This switch does not cause any problems since if $E$ is an abstract free
group on the same finite generating set as $F$ then one can show that $E/P_c(E) \cong F/P_c(F)$ for all $c \geq 1$. If $G$ has $p$-class $c$ then this isomorphism can be used to show that there is a one-to-one correspondence between the normal subgroups $M$ of $E$ with $E/M \cong G$  and the (open) normal subgroups $N$ of $F$ with $F/N \cong G$. The subgroups $M$ and $N$ are free with the same generator ranks since the Schreier index formula applies in both the abstract and pro-$p$ setting. It follows that $M/M^* \cong N/N^*$ and so the definition of the $p$-multiplicator rank is independent of whether one uses an abstract or pro-$p$ presentation for $G$.
Furthermore, one can see that $E/M^* \cong F/N^*$ since both quotients are finite $p$-groups which are $p$-covering groups for $G$. Thus the definition of the nucleus is also independent of whether one uses an abstract or pro-$p$ presentation for $G$.
\end{remark}

A \SSA group of $p$-class $c$ which is terminal has no proper
descendants but must be $H/P_c(H)$ for some Schur $\sigma$-group $H$;
hence it is a Schur $\sigma$-group.  Thus, terminal \SSA groups are
always Schur $\sigma$-groups.  In the other direction, in the
appendix, Blackhurst proves that a non-cyclic $p$-group with trivial
Schur multiplier must be terminal; this is a result to which several
authors have referred, but there appears to be no proof in the
literature.  Since Schur $\sigma$-groups satisfy $r(G)=d(G)$, they
have trivial Schur multiplier, hence finite non-cyclic Schur $\sigma$-groups are
terminal.  In summary, terminal \SSA groups are precisely finite non-cyclic Schur $\sigma$-groups.

\subsection{Measures of $p$-groups}
Let $F_c = F/P_c(F)$  where $c$ is any positive integer with GI-automorphism $\sigma$ induced by the GI-automorphism $\sigma$ on $F$ defined previously. As an analogue of $X \subset \Phi(F)$, we introduce $X_c \subset \Phi(F_c)$ by defining
$$ X_c = X(\Phi(F_c),\sigma) = \{s \in \Phi(F_c)  \mid  \sigma(s) = s^{-1} \}.   $$
Let $G$ be a finite $p$-group of $p$-class $c$ with generator and relation ranks both equal to $g$. One can see that $G$ is a quotient of $F_{c'}$ for all $c' \geq c$. We will say that the tuple of elements $v = (t_1,\ldots,t_g)
\in\Phi(F_{c'})^g$ {\em presents $G$} if $F_{c'} / \langle v \rangle \cong
G$ where $ \langle v \rangle$ denotes the closed normal subgroup of $F_{c'}$
generated by $t_1,\ldots,t_g$. 
We let $S_{c'} = S_{c'}(G)$ denote the set of all such tuples in $\Phi(F_{c'})^g$.  
If $G$ is a \SSA then from Koch-Venkov's results on
presentations of Schur $\sigma$-groups discussed earlier, it follows
that there always exists a tuple $w \in X_{c'}^g$ which presents $G$. 
Just take the image of a tuple of relations in $X^g$ for a Schur $\sigma$-group descended from $G$ under the map induced by the natural projection from $F$ to $F_{c'}$.
We let $T_{c'} = T_{c'}(G)$ denote the set of all such tuples in $X_{c'}^g$. 
Observe that $T_{c'} \subseteq S_{c'}$. 

\begin{definition}
Let $G$ be a \SSA of $p$-class $c$ and generator rank $g$. For $c' \geq c$, we define the {\em $c'$-measure of $G$} by
\[ \mathrm{Meas}_{c'}(G) = \frac{|T_{c'}|}{|X_{c'}|^g}.  \]
\end{definition}
We view the $c'$-measure of a \SSA group $G$ as the probability with which that group arises as a quotient of $F_{c'}$ when one selects
a tuple of relations at random from $X_{c'}^g$.  Shortly, we will examine the sequence $( \mathrm{Meas}_{c'}(G) )_{c' \geq c} $.

\begin{example}\label{ex-p3g2c2}
As an example, let $p=3$ and consider the case where $g = 2$ and $c' = c = 2$.
O'Brien's algorithm yields seven finite $2$-generated
$3$-groups of $3$-class $2$, of which three are \SSA groups. In this case, $F_2 = F/P_2(F)$
has order $3^5$ and we calculate that the set $X_2$ is an elementary
abelian subgroup of order $9$.  Of these three \SSA groups, the one of
order $27$ - call it $G_1$ - arises when the ordered $2$-tuple taken
from $X_2$ generates $X_2$. This happens for $48$ of the $81$ ordered
$2$-tuples.  Thus $\mathrm{Meas}_2(G_1) = 16/27$.  The second group,
of order $81$ - call it $G_2$ - arises when the ordered $2$-tuple
generates one of the $4$ subgroups of $X_2$ of order $3$. Each of
these four subgroups is generated by 8 of the $81$ ordered $2$-tuples in
$X_2 \times X_2$; hence, $\mathrm{Meas}_2(G_2) = 32/81$.  The third
group, of order $243$ - call it $G_3$ - is $F_2$ itself and arises
when both entries in the $2$-tuple are trivial. Therefore,
$\mathrm{Meas}_2(G_3) = 1/81$. Note that $\mathrm{Meas}_3$ of each
of these groups is $0$. An explanation for this will be given shortly.
\end{example}

\begin{remark}
In the above example, $X_{c'}$ happened to be a subgroup; in general, $X$
and $X_{c'}$ are not subgroups.
\end{remark}
 
\begin{lemma}\label{phi-map1}
For all  $d \geq 1$, we have $X_d = X_d'$ where 
\[ X_d' =  \{t^{-1}\sigma(t) \mid  t \in \Phi(F_d)\}. \]
Hence, for all $g \geq 1$, the map $\phi_{d}: \Phi(F_{d})^g
\rightarrow X_{d}^g$ defined by $(t_1,\ldots,t_g) \mapsto (t_1^{-1}
\sigma(t_1),\ldots, t_g^{-1}\sigma(t_g))$ is surjective.  Furthermore,
for each $w \in X_{d}^g$, the fiber $\phi_{d}^{-1}(w)$ is a coset of
$\A_{d}^g$ in $\Phi(F_{d})^g$ where $\A_{d} = \A(F_{d},\sigma)$.
\end{lemma}
\begin{proof}
It is easy to verify that $X_d' \subseteq X_d$. For the reverse
direction, we consider the map $X_d \rightarrow X_d'$ defined by $t
\mapsto t^{-1}\sigma(t) = t^{-2}$. This map is injective since $p$ is
odd and $F_d$ is a finite $p$-group.  It follows that $|X_d| \leq
|X_d'|$ and hence we must have equality $X_d = X_d'$.

The statement that the fibers of $\phi_d$ are cosets is straightforward and makes
use of the fact that $\A_d \subseteq \Phi(F_d)$.
\end{proof}

\begin{remark}
Using the fact $X_d = X_d'$ for all $d$, one can now show that $X =
X'$ where $X' = \{ t^{-1} \sigma(t) \mid t \in \Phi(F) \}$. Since both
sets are closed in $F$, it suffices to prove that $\psi_d(X) =
\psi_d(X')$ for all $d$ where $\psi_d: F \rightarrow F_d$ is the
natural projection. It is easy to see that $X' \subseteq X$ and hence
$\psi_d(X') \subseteq \psi_d(X)$. It follows that
\[     X_d' = \psi_d(X') \subseteq \psi_d(X) \subseteq X_d = X_d' \]
and hence the two the middle containments are also equalities.
\end{remark}

\begin{theorem}\label{thm-meas-relations}
Let $G$ be a \SSA group of $p$-class $c$.
\begin{itemize}
\item[(i)]  We have
\[  \Meas_c(G) = \Meas_{c+1}(G) + \sum_Q \Meas_{c+1}(Q) \]
where the summation is over all immediate descendants $Q$ of $G$ which
are \SSA groups.
\item[(ii)] $\Meas_{c'}(G) = \Meas_{c+1}(G)$ for all $c' \geq c + 1$.
\end{itemize}
\end{theorem}
\begin{proof} 
It follows from Lemma~\ref{phi-map1} that fibers over individual
elements for the maps $\phi_c: \Phi(F_c)^g \rightarrow X_c^g$ and
$\phi_{c+1}: \Phi(F_{c+1})^g \rightarrow X_{c+1}^g$ are uniform in
size. The same statement holds for the natural projection $\psi:
\Phi(F_{c+1})^g \rightarrow \Phi(F_c)^g$. We have an induced map
$\psibar: X_{c+1}^g \rightarrow X_c^g$ obtained by restricting $\psi$
to the subset $X_{c+1}^g \subseteq \Phi(F_{c+1})^g$. It is also
surjective and must have fibers that are uniform in size since
$\psibar \circ \phi_{c+1} = \phi_c \circ \psi$. Thus, we have
\[ 
\Meas_c(G) = \frac{|T_c|}{ |X_c|^g} = \frac{|\psibar^{-1}(T_c)|}{
  |\psibar^{-1}(X_c^g)|} = \frac{|\psibar^{-1}(T_c)|}{ |X_{c+1}|^g}.
\]
The statement in part (i) will follow once we show $\psibar^{-1}(T_c)
= \psibar^{-1}(T_c(G)) = T_{c+1}(G) \cup \bigcup_Q(T_{c+1}(Q))$ where
$Q$ runs through the immediate descendants of $G$. Note that the union
is disjoint by definition of $T_{c+1}$ and $T_{c+1}(Q) = \emptyset$ if
$Q$ is not a \SSA group.

We now check containment in both directions. If $w \in T_{c+1}(G) \cup
\bigcup_Q(T_{c+1}(Q))$ and $\langle w \rangle$ is the normal subgroup
of $F_{c+1}$ generated by $w$, then $H = F_{c+1}/\langle w \rangle$ is
isomorphic to $G$ or an immediate descendant $Q$.  In either case, $H
/ P_c(H) \cong G$ and so $\psibar(w) \in T_c(G)$.  This follows since
the normal subgroup $\langle \psibar(w) \rangle$ in $F_c$ is equal to
the image of the normal subgroup $\langle w \rangle$ under the natural
epimorphism $F_{c+1} \rightarrow F_c$, and thus
\[ F_c/\langle \psibar(w) \rangle \cong F_{c+1} / 
\langle w \rangle P_c(F_{c+1}) \cong H / P_c(H) \cong G. \]

For the other direction, suppose that one has a tuple in $w \in
X_{c+1}^g$ with $\psibar(w) \in T_c(G)$. Then $H = F_{c+1}/\langle w
\rangle$ has $p$-class at most $c+1$ and $H/P_c(H) \cong G$. This last
part follows again since $\langle \psibar(w) \rangle$ is equal to the
image of $\langle w \rangle$ under the natural epimorphism $F_{c+1}
\rightarrow F_c$.  We deduce that $H$ is either $G$ or an immediate
descendant and so by definition $w \in T_{c+1}(G)$ or $w \in
T_{c+1}(Q)$ for some immediate descendant $Q$.

The proof of part (ii) reduces to verifying that
$\psibar^{-1}(T_{c+1}(G)) = T_{c'}(G)$ where $\psibar: X_{c'}^g
\rightarrow X_{c+1}^g$ is the restriction of the natural epimorphism
$\psi: \Phi(F_{c'})^g \rightarrow \Phi(F_{c+1})^g$. Verifying the
containment $T_{c'}(G) \subseteq \psibar^{-1}(T_{c+1}(G))$ is
straightforward. For the reverse direction we must make use of the
assumption that $G$ has $p$-class $c$. Suppose that $w \in X_{c'}^g$
and $\psibar(w) \in T_{c+1}(G)$. We wish to show that $w \in
T_{c'}(G)$. Let $v \in X^g \subseteq F^g$ be a lift of $w$ under the
natural epimorphism $F^g \rightarrow F_{c'}^g$ and consider the Schur
$\sigma$-group $\hat{G} = F/\langle v \rangle$. Let $\hat{G}_d$ denote
the quotient $\hat{G} / P_d(\hat{G})$. Then we have $\hat{G}_{c+1}
\cong G \cong \hat{G}_c$ since $\psibar(w) \in T_{c+1}(G)$ and $G$ has
$p$-class $c$. Equivalently, $P_{c+1}(\hat{G}) = P_c(\hat{G})$. An
inductive argument now shows that $\hat{G}_{d} \cong G$ for all $d
\geq c$. In particular, $F_{c'}/\langle w \rangle \cong \hat{G}_{c'}
\cong G$ which shows that $w \in T_{c'}(G)$ as desired.
\end{proof}

\begin{remark}
We have $\Meas_1(G) = 1$ when $G$ is the elementary abelian $p$-group
of generator rank $g$. One can now use
Theorem~\ref{thm-meas-relations} to see that, for each $c \geq 1$,
$\Meas_c(G)$ defines a discrete probability measure on the set of
maximal $p$-class $c$ quotients of all Schur $\sigma$-groups of
generator rank $g$. This finite set of groups consists of the Schur
$\sigma$-ancestor groups of $p$-class exactly $c$, together with all
Schur $\sigma$-groups of $p$-class less than $c$.
\end{remark}

\begin{definition}
Let $G$ be a \SSA group of $p$-class $c$. We define the {\em measure
  of $G$} (denoted $\Meas(G)$) to be the common value of
$\Meas_{c'}(G)$ for $c' \geq c+1$.
\end{definition}

\begin{theorem}
\label{remark-meas-relations}
Let $G$ be a \SSA group of $p$-class $c$.  
\begin{enumerate}
\item If $G$ is a non-cyclic Schur $\sigma$-group, then $\Meas(G)=\Meas_c(G)>0$. 
\item If $G$ is not a Schur $\sigma$-group, then $\Meas(G)=0$ and $\Meas_c(G)$ is the sum of 
the $c+1$-measures of its immediate descendants.
\end{enumerate}
\end{theorem}
\begin{proof}
If $G$ is a non-cyclic Schur $\sigma$-group of $p$-class $c$ then, as discussed
in Section~\ref{subsection-prelim}, it has no descendants and so by
part (i) of Theorem~\ref{thm-meas-relations} we see that $\Meas_c(G) =
\Meas_{c+1}(G)$. It follows that $\Meas(G) = \Meas_c(G) > 0$ since
$T_c(G) \neq \emptyset$.

On the other hand, if $\Meas(G) = \Meas_{c+1}(G) > 0$ then $T_{c+1}(G)
\neq \emptyset$. Let $w \in T_{c+1}(G)$ and consider a lift $u \in
X^g$ and the Schur $\sigma$-group $\hat{G} = F/\langle u \rangle$. The
arguments in the proof of part (ii) of
Theorem~\ref{thm-meas-relations} now show that $G \cong \hat{G}_c
\cong \varprojlim \hat{G}_{c'} = \hat{G}$. Hence $G$ itself is a Schur
$\sigma$-group of $p$-class $c$. Thus if $G$ is a \SSA of $p$-class
$c$ which is not a Schur $\sigma$-group then $\Meas_{c+1}(G) = 0$ and
so $\Meas_c(G)$ is the sum of the $c+1$-measures of its immediate
descendants by part (i) of Theorem~\ref{thm-meas-relations},
\end{proof}

\subsection{Measures of abelian $p$-groups}
\label{section-abelian}

We are now going to define analogous measures on the class of finite
abelian $p$-groups and relate these to the measures introduced
above. This will be used to justify the assertion that our conjectures
in the non-abelian setting generalize the Cohen-Lenstra heuristics for
$p$-class groups.

In what follows, the role of $F$ and $F_c$ will be played by the
abelianizations $F^{ab}$ and $F_c^{ab}$. Note that $(F_c)^{ab} \cong
(F^{ab})_c$.  Every abelian pro-$p$ group $G$ comes equipped with a
unique $\sigma$-automorphism, namely the inversion mapping $x \mapsto
x^{-1}$.  We define sets $X^{ab}$ and $X_c^{ab}$ in an analogous way
to $X$ and $X_c$ but things are now simpler and it is easy to verify
that $X^{ab} = \Phi(F^{ab})$ and $X_c^{ab} = \Phi(F_c^{ab})$.

Let $G$ be a finite abelian $p$-group of $p$-class $c$ with generator
rank $g$ and let $c' \geq c$.  We will say that the tuple of elements
$v = (t_1,\ldots,t_g) \in\Phi(F_{c'}^{ab})^g$ {\em presents $G$} if
$F_{c'}^{ab} / \langle v \rangle \cong G$ where $ \langle v \rangle$
denotes the (normal) subgroup of $F_{c'}^{ab}$ generated by
$t_1,\ldots,t_g$. Such tuples must exist since $G$ is finite.  We let
$S_{c'}^{ab} = S_{c'}^{ab}(G)$ denote the set of all such tuples in
$\Phi(F_{c'}^{ab})^g$.  In the non-abelian setting, we introduced a
second set of tuples $T_{c'} \subseteq S_{c'}$.  We can do the same in
the abelian setting, but the situation now is simpler and we have
$T_{c'}^{ab}$ = $S_{c'}^{ab}$ since $X_{c'}^{ab} = \Phi(F_{c'}^{ab})$.

\begin{definition}
Let $G$ be an abelian $p$-group of $p$-class $c$ and generator rank
$g$. For $c' \geq c$, we define the {\em abelian $c'$-measure of $G$}
by
\[ \mathrm{Meas}_{c'}^{ab}(G) = \frac{|T_{c'}^{ab}|}{|X_{c'}^{ab}|^g}  \left(=  \frac{|S_{c'}^{ab}|}{|\Phi(F_{c'}^{ab})|^g}\right).  \]
\end{definition}
We view the abelian $c'$-measure of a finite $p$-group $G$ as the
probability with which that group arises as a quotient of
$F_{c'}^{ab}$ when one selects a tuple of relations at random from
$(X_{c'}^{ab})^g = \Phi(F_{c'}^{ab})^g$.

\begin{theorem}\label{thm-meas-relations-ab}
Let $G$ be an abelian $p$-group of $p$-class $c$.
\begin{itemize}
\item[(i)]  We have
\[  \Meas_c^{ab}(G) = \Meas_{c+1}^{ab}(G) + \sum_Q \Meas_{c+1}^{ab}(Q) \]
where the summation is over all immediate abelian descendants $Q$ of $G$.
\item[(ii)] $\Meas_{c'}^{ab}(G) = \Meas_{c+1}^{ab}(G)$ for all $c' \geq c + 1$.
\end{itemize}
\end{theorem}
\begin{proof} The proof is carried out in exactly the same fashion as the proof of Theorem~\ref{thm-meas-relations}. We omit the details.
\end{proof}

\begin{definition}
Let $G$ be an abelian $p$-group of $p$-class $c$. We define the {\em abelian measure of $G$} (denoted $\Meas^{ab}(G)$) to be the common value of $\Meas_{c'}^{ab}(G)$ for $c' \geq c+1$.
\end{definition}
\begin{remark}
It follows from part (i) of Theorem~\ref{thm-meas-relations-ab} that if $G$ is an abelian $p$-group of $p$-class $c$ then
\[ \Meas^{ab}(G) = \Meas_c^{ab}(G) - \sum_Q \Meas_{c+1}^{ab}(Q) \]
where the summation is over all abelian groups $Q$ of $p$-class $c+1$ with $Q/Q^{p^c} \cong G$;  here $Q^{p^c}$ is the subgroup of $Q$ generated by all $p^c$-th powers.
\end{remark}
The following theorem and its corollary provide the link between $\Meas_{c}^{ab}$ and $\Meas_c$.

\begin{theorem}
Let $G$ be an abelian $p$-group of $p$-class $c$. For all $c' \geq c$ we have
\[ \Meas_{c'}^{ab}(G) = \sum_Q \Meas_{c'}(Q) \]
where the summation is over all \SSA groups $Q$ with $p$-class at most $c'$ and $Q^{ab} \cong G$.
\end{theorem}
\begin{proof} 
It follows from  Lemma~\ref{phi-map1} that the fibers of the map $\phi_{c'}: \Phi(F_{c'})^g \rightarrow
X_{c'}^g$  are uniform in size. The same statement holds for the analogous map on the abelian side, namely
$\phi_{c'}^{ab}: \Phi(F_{c'}^{ab})^g \rightarrow (X_{c'}^{ab})^g = \Phi(F_{c'}^{ab})^g$ given by $(t_1,\ldots, t_g) \mapsto
(t_1^{-1} \sigma(t_1), \ldots, t_g^{-1} \sigma(t_g)) = (t_1^{-2}, \ldots, t_g^{-2})$. Indeed, the latter map is a bijection since $p$ is odd. We also have a projection map $\psi: \Phi(F_{c'})^g \rightarrow \Phi(F_{c'}^{ab})^g$ and its restriction $\psibar: X_{c'}^g \rightarrow (X_{c'}^{ab})^g = \Phi(F_{c'}^{ab})^g$.  Since the projection $\psi$ has uniform fibers and $\psibar \circ \phi_{c'} = \phi_{c'}^{ab} \circ \psi$, we see that $\psibar$ is also onto and has uniform fibers. Thus
\[ 
\Meas_{c'}^{ab}(G) = \frac{|T_{c'}^{ab}|}{ |X_{c'}^{ab}|^g} = \frac{|\psibar^{-1}(T_{c'}^{ab})|}{
  |\psibar^{-1}((X_{c'}^{ab})^g)|} = \frac{|\psibar^{-1}(T_{c'}^{ab})|}{ |X_{c'}|^g}.
\]
We have $(F_{c'}/\langle w \rangle)^{ab} \cong F_{c'}^{ab}/ \langle \psibar(w) \rangle$ for all $w \in X_{c'}^g$. 
If $w \in \psibar^{-1}(T_{c'}^{ab}) = \psibar^{-1}(T_{c'}^{ab}(G))$ then by definition  $F_{c'}^{ab}/ \langle \psibar(w) \rangle \cong G$ and so $(F_{c'}/\langle w \rangle)^{ab} \cong G$ 
which means $w \in T_{c'}(Q)$ for the \SSA group $Q = F_{c'}/\langle w \rangle$ and we have $Q^{ab} \cong G$. Conversely, the same isomorphisms show that if $w \in  T_{c'}(Q)$ for a \SSA group $Q$ with $Q^{ab} \cong G$ then $w \in  \psibar^{-1}(T_{c'}^{ab}(G))$. Thus we have
$\psibar^{-1}(T_{c'}^{ab}(G)) = \bigcup_Q T_{c'}(Q)$ where $Q$ runs through the \SSA groups of $p$-class at most $c'$ with $Q^{ab} \cong G$. The union is disjoint so the statement about the measures now follows.
\end{proof}
If $Q^{ab} \cong G$ and $G$ has $p$-class $c$ then $Q$ must have $p$-class at least $c$. We thus have the following corollary.
\begin{corollary}
Let $G$ be an abelian $p$-group of $p$-class $c$. Then
\[ \Meas_c^{ab}(G) = \sum_Q \Meas_c(Q) \]
where the summation is over all \SSA groups $Q$ with $p$-class exactly $c$ that satisfy $Q^{ab} \cong G$.
\end{corollary}

\subsection{Formulas for $\Meas_c^{ab}$ and  $\Meas_c$}
\label{subsection-formulas}

We will now derive formulas for the various measures introduced so far. The formula
for $\Meas^{ab}$ in the next theorem can be found as Theorem 6.3, p.49 of \cite{CL2}. 
We give a detailed derivation here in order to lay the groundwork for the proof of Theorem~\ref{main-group-theory} which is structured the same way and begins with the same counting argument.

\begin{theorem}\label{main-group-theory-ab}
Let $G$ be an abelian $p$-group of $p$-class $c$ and generator rank $g$. We have
\begin{equation*}
\Meas_c^{ab}(G)  
= \frac{1}{|\Aut(G)|} \,
p^{g^2}
\prod_{k=1}^g (1 - p^{-k}) \prod_{k={1 + g - u}}^g (1 - p^{-k})
\end{equation*}
where $u$ counts the number of cyclic groups of order strictly less than $p^c$ in the direct product decomposition of $G$. \\
For $c' > c$, we have
\begin{eqnarray*}
\Meas_{c'}^{ab}(G) = \Meas^{ab}(G) = 
\frac{1}{|\Aut(G)|}
\,
p^{g^2}
\prod_{k=1}^g (1 - p^{-k})^2 .
\end{eqnarray*}
\end{theorem}
\begin{proof}
To compute $\Meas_c^{ab}(G)$ we need to count tuples of relations in $\Phi(F_c^{ab})^g$ which present $G$. We will do this in two stages
by following the same strategy as in~\cite{B}. First, we will count the number of normal subgroups $\Rbar$ in $F_c^{ab}$ with $F_c^{ab}/\Rbar \cong G$ by counting certain collections of epimorphisms. Then we will count the number of generating tuples that generate each such subgroup as a normal subgroup although the normality condition imposes no restriction here since $F_c^{ab}$ is abelian.

Let $\mathrm{Epi}(F,G)$ be the set of epimorphisms from $F$ to $G$ where $F$ is the free pro-$p$ group
on $g$ generators.
Such epimorphisms are in one-to-one correspondence with ordered
$g$-tuples of elements in $G$ that generate $G$.  By Burnside's basis theorem, 
a tuple of elements generates $G$ if and only if it generates $G/\Phi(G)$. It follows that
\[ |\mathrm{Epi}(F,G)| = |\Phi(G)|^g \,
(p^g-p^{g-1})(p^g-p^{g-2}) \dots (p^g-1) = |\Phi(G)|^g \prod_{k=1}^g (p^g - p^{g-k})\]
since $G/\Phi(G)$ is an $\mathbb{F}_p$-space of dimension $g$.

Two epimorphisms have the
same kernel if and only if they differ by an automorphism of $G$, so
dividing by $|\Aut(G)|$ gives the number of (closed) normal subgroups
$R$ of $F$ with quotient isomorphic to $G$. Since $G$ is abelian and has $p$-class $c$ we have
 $P_c(F)[F,F] \subseteq R$ for each such subgroup $R$ and there is 
a one-to-one correspondence between these subgroups of $F$
and the subgroups $\overline{R}$ of $F_c^{ab}$ such that $F_c^{ab}/\overline{R} \cong G$. Thus the number of such subgroups
$\Rbar$ is
\[ \frac{|\mathrm{Epi}(F,G)|}{|\Aut(G)|} =\frac{ |\Phi(G)|^g}{|\Aut(G)|} \prod_{k=1}^g (p^g - p^{g-k}). \]

Now we need to count how many $g$-tuples of elements generate each $\Rbar$ as a (normal) subgroup of $F_c^{ab}$.
A $g$-tuple of elements generates $\overline{R}$ as a subgroup of $F_c^{ab}$ if
and only if their images generate the $\mathbb{F}_p$-space $V = \Rbar/\Phi(\Rbar)$. Since $F_c^{ab} \cong F^{ab}/(F^{ab})^{p^c}$ is a product of $g$ copies of $\mathbb{Z}_{p^c}$, the dimension of $V$ is equal to the number of cyclic factors in the decomposition of the abelian group $G$ which are strictly smaller than $\mathbb{Z}/{p^c\mathbb{Z}}$. This is the quantity $u$ in the statement of the theorem. There are $\prod_{k=1}^u (p^g - p^{u-k})$ $g$-tuples of elements in $V$ which span this space and hence 
\[   |\Phi(\Rbar)|^g \prod_{k=1}^u (p^g - p^{u-k}) \]
$g$-tuples that generate each subgroup $\Rbar$.  Note that $|\Phi(\Rbar)| = |F_c^{ab}|/[F_c^{ab} : \Phi(\Rbar)] = |F_c^{ab}| / (|G| p^u)$ so this quantity is independent of the particular subgroup $\Rbar$ being considered.

Combining the statements above, we have
\begin{eqnarray*}
\mathrm{Meas}_c^{ab}(G) = \frac{|S_c^{ab}(G)|}{|\Phi(F_c^{ab})|^g} 
&=& \frac{1}{|\Phi(F_c^{ab})|^g}  \frac{ |\Phi(G)|^g}{|\Aut(G)|} \prod_{k=1}^g (p^g - p^{g-k}) |\Phi(\Rbar)|^g \prod_{k=1}^u (p^g - p^{u-k}) \\
&=& \frac{1}{|\Phi(F_c^{ab})|^g}  \frac{ (|\Phi(F_c^{ab})|/|\Rbar|)^g}{|\Aut(G)|} \prod_{k=1}^g (p^g - p^{g-k}) \frac{|\Rbar|^g}{p^{gu}} \prod_{k=1}^u (p^g - p^{u-k}) \\
&=& \frac{1}{|\Aut(G)|} \,
\frac{1}{p^{g u}}
\prod_{k=1}^g (p^g - p^{g-k}) \prod_{k=1}^u (p^g - p^{u-k}) \\
&=& \frac{1}{|\Aut(G)|} \,
p^{g^2}
\prod_{k=1}^g (1 - p^{-k}) \prod_{k={1 + g - u}}^g (1 - p^{-k})
\end{eqnarray*}

The second statement about $\Meas_{c'}^{ab}(G)$ for $c' > c$ is verified in exactly the same way. The only difference occurs in the second step. One sees that the space $V = \Rbar / \Phi(\Rbar)$ has dimension $g$ since $G$ has $p$-class $c$ which means that all $g$ of its cyclic components are strictly smaller that $\mathbb{Z}/{p^{c'} \mathbb{Z}}$. Thus the formula one obtains is the one above with $u = g$.
\end{proof}

\begin{remark}\label{remark-eta}
If we define $\eta_j(p) = \prod_{k=1}^j (1 - p^{-k})$ as in~\cite{CL2}, then the formulas in Theorem~\ref{main-group-theory-ab} can be written
\begin{eqnarray*}
\mathrm{Meas}_c^{ab}(G) &=& 
\frac{1}{|\Aut(G)|} \,
p^{g^2} \left(\frac{ \eta_g(p)^2}{\eta_{g-u}(p) }\right) \\
\mathrm{Meas}^{ab}(G) &=& \frac{1}{|\Aut(G)|} \, p^{g^2} \eta_g(p)^2.
\end{eqnarray*}
\end{remark}

To derive similar formulas for the measures in the non-abelian context, we need an additional technical assumption on the groups involved. 
Recall that $F$ is the free pro-$p$ group of generator rank $g$. Let $G$ be a \SSA group of $p$-class $c$ with generator rank $g$. Given $w \in T_c(G)$, the normal subgroup $\langle w \rangle$ is the kernel of an epimorphism from $F_c$ to $G$ and satisfies $\sigma(\langle w \rangle) = \langle w \rangle$. In the lemma and
theorems which follow, we will need to make the much stronger
assumption that the kernel of {\em every} epimorphism from $F_c$ to
$G$ is invariant under $\sigma$. Or, equivalently, that the kernel of every epimorphism from $F$ to $G$ is invariant under $\sigma$. 
\begin{definition}
\label{kip}
If $G$ is a finite $p$-group with the same generator rank as the free group $F$ and $\sigma(\ker \psi) = \ker \psi$ for every epimorphism $\psi: F \rightarrow G$ then we will say that $G$ satisfies the {\em kernel invariance property} (KIP). 
\end{definition}
Some additional remarks about this property and its range of applicability will be made later in Section~\ref{subsection-kip}.

\begin{lemma}\label{phi-map2}
Let $G$ be a \SSA group of $p$-class $c$ satisfying KIP. Let $c' \geq c$ and define $\phi_{c'}$ and $\A_{c'}$ as in Lemma~\ref{phi-map1}. The following statements hold.
\begin{itemize}
\item[(i)] If $v \in S_{c'}$ then $\langle \phi_{c'}(v) \rangle = \langle v \rangle$ and $\phi_{c'}(v) \in T_{c'}$.
\item[(ii)] If $w \in T_{c'}$ then there exists $v \in S_{c'}$ such that $\phi_{c'}(v) = w$.
\item[(iii)] Let $v \in S_{c'}$ and $\Rbar = \langle v \rangle$. We have $u \in S_{c'} \cap \phi_{c'}^{-1}(v)$ if and only if  $u v^{-1} \in (\A_{c'} \cap \Rbar)^g$.
\item[(iv)] $[\A_{c'} : \A_{c'} \cap \Rbar] = \see(G)$ where $\Rbar$ is the kernel of
  any epimorphism from $F_{c'}$ to $G$.
\end{itemize}
\end{lemma}
\begin{proof}
First, note that since $G$ has $p$-class $c$, the kernel of every epimorphism from $F_{c'}$ to $G$ where $c'
\geq c$ must be invariant under $\sigma$ since each such epimorphism
is induced by an epimorphism from $F$ to $G$ for which the kernel is invariant by assumption. 
This form of KIP is used below and in some of the later proofs in this section.

For part (i), suppose $v =(t_1,\ldots,t_g) \in S_{c'}$. Then the normal subgroup $\langle v \rangle$
is the kernel of an epimorphism $F_{c'} \rightarrow G$ and so is
invariant under $\sigma$ by assumption. It follows that $t_i^{-1}
\sigma(t_i) \in \langle v \rangle$ for all $i$ and so $\langle \phi_{c'}(v)
\rangle \subseteq \langle v \rangle$. This means there is a natural
epimorphism $F_{c'}/\langle \phi_{c'}(v) \rangle \rightarrow F_{c'}/\langle v
\rangle \cong G$. Applying Lemma~4.10 in~\cite{BE} to this epimorphism, we see
that it must be an isomorphism and so $\langle \phi_{c'}(v) \rangle =
\langle v \rangle$ which means $\phi_{c'}(v)$ also presents $G$.

For part (ii), we again make use of the fact that $p$ is odd. We
restrict the map $\phi_{c'}$ to $X_{c'}^g \subseteq \Phi(F_{c'})^g$. One
observes that the restriction $\phi: X_{c'}^g \rightarrow X_{c'}^g$ is
the powering map $t \mapsto t^{-2}$ in each component. If one chooses
$n$ such that $(-2)^n \equiv 1$ modulo the exponent of the group
$F_{c'}$ then the $n$th iterate of this map is the identity. Starting
with any tuple $w \in T_{c'} \subseteq X_{c'}^g$  we then have $w =
\phi_{c'}^n(w) = \phi_{c'}( \phi_{c'}^{n-1}(w))$. Taking $v = \phi_{c'}^{n-1}(w)$ we note
that $v$ must also present $G$ by repeated application of part (i).

For part (iii), let $v =(t_1,\ldots,t_g) \in S_{c'}$. We can apply part (i) to see that if $u \in S_{c'}$  
and $\phi_{c'}(u) = \phi_{c'}(v)$ then $\langle u \rangle = \langle \phi_{c'}(u)
\rangle = \langle \phi_{c'}(v) \rangle = \langle v \rangle = \Rbar$. By
Lemma~\ref{phi-map1}, we can write $u = (y_1 t_1,\ldots, y_g t_g)$ with $y_i \in
\A_{c'}$ for all $i$. Combining the previous two statements, we deduce that
$y_i \in \A_{c'} \cap \Rbar$ for all~$i$ and so $u v^{-1} \in (\A_{c'} \cap \Rbar)^g$.

Conversely, if we let $u = (y_1 t_1,\ldots, y_g t_g)$ for any
$g$-tuple of elements $(y_1,\ldots,y_g) \in (\A_{c'} \cap \Rbar)^g$ then
$\phi_{c'}(u) = \phi(v)$ and  $u$ can be seen to present $G$ as follows.
That $t_1,...,t_g$ generate $\overline{R}$ as a normal subgroup of
$F_c$ is equivalent to their images spanning the $\mathbb{F}_p$-space
$\Rbar/\Rbar^* \cong R/ P_c(F) R^*$. 
We note that the induced action of $\sigma$ on this vector space
is entirely by inversion. This follows by first using \cite{Gru}
p.100, Prop. 4, to identify the vector space with $H_2(G,\F_p)$.
Next, consider the homology long exact sequence associated to the
short exact sequence
$$ 0 \rightarrow \Z \rightarrow \Z \rightarrow \Z/p \rightarrow 0 $$

This is $$ ... \rightarrow H_2(G,\Z) \rightarrow H_2(G,\Z) \rightarrow
H_2(G,\Z/p) \rightarrow H_1(G,\Z) \rightarrow H_1(G,\Z) \rightarrow ... $$
which yields the exact sequence
$$ 0 \rightarrow H_2(G,\Z)/pH_2(G,\Z) \rightarrow H_2(G,\Z/p) \rightarrow
H_1(G,\Z)[p] \rightarrow 0 $$
These maps are $\sigma$-equivariant and the 3rd and 4th terms have the same dimension 
over $\F_p$, implying that $H_2(G,\Z/p)$ is $\sigma$-isomorphic to
$H_1(G,\Z)[p]$. Since $G$ is finite, this in turn is $\sigma$-isomorphic
to $H_1(G,\Z/p)$, which by \cite{Gru}, p.99, Prop. 3, is $\sigma$-isomorphic to
$G/\Phi(G)$, on which $\sigma$ acts entirely by inversion.
So if $y_1,\ldots, y_g$ lie in $\A_{c'} \cap \Rbar$ then their images in
$\Rbar/\Rbar^*$ must be trivial, and so the images of the $y_i t_i$ for 
$i=1,\ldots, g$ will also span this space. It follows that $\langle u
\rangle = \langle v \rangle = \Rbar$ and so $u$ also presents $G$.
Thus we have $u \in S_{c'} \cap \phi_{c'}^{-1}(v)$.

Finally, part (iv) follows since if $G = F_{c'}/\Rbar$ then $\A_{c'}/ \A_{c'} \cap
\Rbar \cong \A_{c'} \Rbar / \Rbar = \A(G,\tau)$ where $\A(G,\tau)$ is defined with
respect to the $\sigma$-automorphism $\tau$ on $G$ induced by the
$\sigma$-automorphism on $F_{c'}$. An induced automorphism exists
since $\sigma(\Rbar) = \Rbar$ by assumption. To verify the last
equality, one checks containment in both directions. First, since $\A_{c'} =
\A(F_{c'}, \sigma)$ it is easy to see that $\A_{c'} \Rbar/ \Rbar \subseteq \A(G,\tau)$. For
the reverse containment, suppose that $x \in F_{c'}$ represents an
element $g \in \A(G,\tau)$, then $\sigma(x) = xr$ for some $r \in
\Rbar$. Since $x = \sigma^2(x) = x r \sigma(r)$ one sees that
$\sigma(r) = r^{-1}$. Using the fact that the map $s \mapsto s^2$ is a
bijection from $\Rbar$ to $\Rbar$ we can select $s \in \Rbar$ such
that $s^2 = r$. One can then verify that $x' = x s \in \A_{c'}$ and hence $g
= x \Rbar = x' \Rbar \in \A_{c'} \Rbar/ \Rbar$.

Although the set $\A(G,\tau)$ does depend on the choice of GI-automorphism
$\tau$, its size $\see(G)$ does not, as explained in Section~\ref{subsection-prelim}.
\end{proof}

Before stating the next theorem we need to define one additional quantity.
\begin{definition}
Let $G$ be a finite $p$-group. Define $\r(G)$ to be $p$-multiplicator rank of $G$ minus the nuclear rank of $G$. Equivalently, $\r(G)$ is the dimension of the $\mathbb{F}_p$-space $R/P_c(F) R^*$ where $G \cong F/R$.
\end{definition}
We note that for any finite $p$-group $G$ we have $\r(G) \geq 0$.  It is a fact that $r(H) \geq \r(G)$
 for any descendant $H$ of $G$ (Prop.~2 of \cite{BN}). In particular, if $G$ is $g$-generated and $\r(G) > g$ then $G$ and its descendants cannot be \SSA groups.

\begin{theorem}\label{main-group-theory}
Let $G$ be a \SSA group of $p$-class $c$ and rank $g$ satisfying KIP. Let $\r=\r(G)$ and $r = r(G)$. Then
\begin{eqnarray*}
\Meas_c(G) = 
\frac{\see(G)^g}{|\Aut(G)|}
\,
p^{g^2}
\prod_{k=1}^g (1 - p^{-k}) \prod_{k={1 + g - \r}}^g (1 - p^{-k})
\end{eqnarray*}
and for $c' > c$
\begin{eqnarray*}
\Meas_{c'}(G) = 
\frac{\see(G)^g}{|\Aut(G)|}
\,
p^{g^2}
\prod_{k=1}^g (1 - p^{-k}) \prod_{k={1 + g - r}}^g (1 - p^{-k})
\end{eqnarray*}
\end{theorem}

\begin{proof}
To compute $\Meas_c(G)$, we will first find the proportion of $g$-tuples of relators in $\Phi(F_c)$ that present $G$.
We will then modify this to obtain $\mathrm{Meas}_{c}(G)$. A similar argument yields the second formula.

For the first step, we use similar arguments as in Theorem~\ref{main-group-theory-ab}.
If $G = F/R$ has $p$-class $c$ then $P_c(F) \subseteq R$ and we have
a one-to-one correspondence between the normal subgroups $R$ of $F$
such that $F/R \cong G$ and the normal subgroups $\overline{R}$ of
$F_c$ such that $F_c/\overline{R} \cong G$. The number of such normal subgroups  is
\[ \frac{|\mathrm{Epi}(F,G)|}{|\Aut(G)|} =\frac{ |\Phi(G)|^g}{|\Aut(G)|} \prod_{k=1}^g (p^g - p^{g-k}). \]

A $g$-tuple of elements generates $\Rbar = R/P_c(F)$ as a normal subgroup of $F_c$ if
and only if its image generates the  $\mathbb{F}_p$-space $V = R / P_c(F) R^*$. This has dimension $\r$ by definition.
If we let $\Rbar^* = P_c(F)R^*/P_c(F) \subseteq F_c$ then the number of $g$-tuples that generate $\Rbar$ is 
\[   |\Rbar^*|^g \prod_{k=1}^\r (p^g - p^{\r-k}). \]
A similar calculation to the one in Theorem~\ref{main-group-theory-ab} now shows that 
\begin{eqnarray*}
\frac{|S_c(G)|}{|\Phi(F_c)|^g} 
&=& \frac{1}{|\Phi(F_c)|^g}  \frac{ |\Phi(G)|^g}{|\Aut(G)|} \prod_{k=1}^g (p^g - p^{g-k}) |\Rbar^*|^g \prod_{k=1}^\r (p^g - p^{\r-k}) \\
&=& \frac{1}{|\Aut(G)|}
\,
p^{g^2}
\prod_{k=1}^g (1 - p^{-k}) \prod_{k={1 + g - \r}}^g (1 - p^{-k})
\end{eqnarray*}
where we've made use of the fact that $|\Rbar^*| = |\Rbar|/p^\r$ and $|\Phi(G)| = |\Phi(F_c)|/|\Rbar|$.

We now relate this quantity to $\Meas_c(G)$. Using parts (i), (ii) and (iii) of  Lemma~\ref{phi-map2}, we have 
\[|S_c(G)| = |S_c|  =  |T_c| \cdot |\A_c \cap \Rbar |^g. \] 
It follows that
\begin{eqnarray*}
\Meas_c(G) = \frac{|T_c|}{|X_c|^g} &=& \frac{|S_c|}{|X_c|^g |\A_c
  \cap \Rbar |^g} = \frac{1}{|\A_c \cap \Rbar |^g} \frac{|\Phi(F_c)|^g}{|X_c|^g } \frac{|S_c|}{|\Phi(F_c)|^g} \\ 
  &=& \frac{|\A_c|^g}{|\A_c \cap \Rbar |^g} \frac{|S_c|}{|\Phi(F_c)|^g} = \see(G)^g  \frac{|S_c|}{|\Phi(F_c)|^g}
\end{eqnarray*}
where we have also made use of Lemma~\ref{phi-map1} and part (iv) of
Lemma~\ref{phi-map2} in the last two steps. Substituting our earlier expression for $|S_c| /
|\Phi(F_c)|^g$, we arrive at the
formula for $\Meas_c(G)$ in the statement of the theorem.

This completes the verification of the formula for
$\Meas_c(G)$. The verification of the formula for
$\Meas_{c'}(G)$ where $c' > c$ is almost identical. The only
part that changes is the second step where one now counts the number of
$g$-tuples generating a normal subgroup $\Rbar = R/P_{c'}(F)$ with $F_{c'}/\Rbar \cong G$. 
Since $G$ has $p$-class $c$ we have $P_c(F) \subseteq R$ and
so $P_{c'}(F) \subseteq R^*$ for all $c' > c$. It follows that in this
case $V = R / P_{c'}(F) R^* = R/ R^*$. This is the $p$-multiplicator
whose dimension as an $\mathbb{F}_p$-space is equal to the relation
rank $r$. Thus the formula for the number of $g$-tuples can be
obtained by taking the formula in the first argument and replacing the
quantity $h$ with $r$.
\end{proof}

\begin{corollary}\label{meas-schur-sigma}
Let $G$ be a non-cyclic Schur $\sigma$-group of $p$-class $c$ and rank $g$ satisfying KIP.
Then
\begin{eqnarray*}
\Meas(G) = \Meas_c(G) =
\frac{\see(G)^g}{|\Aut(G)|}
\,
p^{g^2}
\prod_{k=1}^g (1 - p^{-k})^2.
\end{eqnarray*}
\end{corollary}

\begin{example}
Let's compute $\mathrm{Meas}_2(G)$ for the \SSA groups of $3$-class $2$ in
Example~\ref{ex-p3g2c2} using Theorem~\ref{main-group-theory}.  The fact that these three groups
satisfy KIP can be verified computationally or by using Theorem~\ref{kernel-invariance} since the
three groups are all immediate descendants of $F_1 = F/P_1(F)$.  We
have $p = 3$, $g = 2$ and $c = 2$ so the formula reduces to
$\mathrm{Meas}(G_i) = 48 k \see(G_i)^2/|\Aut(G_i)|$, where $k =
\frac{1}{3^{2\r}} \prod_{k=1}^\r (3^2 - 3^{\r-k})$ is the proportion
of ordered pairs of vectors that span an $\r$-dimensional vector space
over $\F_3$ and $\r=\r(G_i)$.  For $G_1, G_2, G_3$ we have
$\see(G_i)=3$ for all $i$, $\r(G_i) = 2,1,0$ respectively, and
$|\Aut(G_i)|=432, 972, 34992$ respectively. Thus, $\mathrm{Meas}(G_1)
= 48 \times 16/27 \times 3^2 \times 1/432 = 16/27$;
$\mathrm{Meas}(G_2) = 48 \times 8/9 \times 3^2 \times 1/972 = 32/81$;
$\mathrm{Meas}(G_3) = 48 \times 1 \times 3^2 \times 1/34992 =
1/81$. These values agree with those obtained by our earlier direct
computations.
\end{example}

\begin{definition}
Suppose $G$ is a finite $p$-group equipped with a GI-automorphism $\tau$.
We denote by $\Aut_\tau(G)$ the set of all automorphisms of $G$ which commute with $\tau$.
\end{definition}

\begin{theorem}\label{aut-group-orders}
Suppose $G$ is a finite $p$-group equipped with a
GI-automorphism $\tau$ and which satisfies KIP.  Then $|\Aut_\tau(G)| =
|\Aut(G)|/\see(G)^g$.
\end{theorem}

\begin{proof}
Let $\Sigma(G)$ be the set of all GI-automorphisms of $G$. The
automorphism group $\Aut(G)$ acts on $\Sigma(G)$ by conjugation and
this action is transitive by Hall's theorem and Schur-Zassenhaus.  The
stabilizer of $\tau \in \Sigma(G)$ is $\Aut_\tau(G)$ so we have
$|\Aut(G)| = |\Aut_\tau(G)| |\Sigma(G)|$. We will now show that
$|\Sigma(G)| = \see(G)^g$ which implies the statement of the theorem.

Consider the set $\mathcal{E}(F,G)$ of epimorphisms from $F$ to
$G$. We are going to count the number of elements in
$\mathcal{E}(F,G)$ in two different ways. Let $\phi \in
\mathcal{E}(F,G)$. The kernel of $\phi$ is invariant under $\sigma$
since $G$ satisfies KIP. It follows that $\sigma$ induces a
GI-automorphism on $G$, which we denote by $\alpha$, satisfying
\[ \alpha(\phi(x)) = \phi(\sigma(x)) \ \ (\ast) \]  
for all $x$ in $F$.  We thus have a map $\mathcal{E}(F,G) \rightarrow
\Sigma(G)$ defined by $\phi \mapsto \alpha$. This map is surjective
due to work of Koch and Venkov discussed in
Section~\ref{subsection-prelim}. To understand the fibers of this map,
we fix $\alpha$ and ask which $\phi$ satisfy $(\ast)$.  First note
that $\phi$ is determined by $(\phi(x_1),...,\phi(x_g)) \in G^g$ and
that any ordered $g$-tuple is possible so long as they generate $G$
and satisfy $(\ast)$.  The property $(\ast)$ says that
$\alpha(\phi(x_i)) = \phi(x_i^{-1}) = \phi(x_i)^{-1}$, in other words,
that $x_i \in X(G, \alpha)$ for all $i$.  Thus every $\phi$ yields an
element of $X(G,\alpha)^g$ generating $G$ and vice versa every element
of $X(G,\alpha)^g$ generating $G$ specifies a legitimate $\phi$. The
size of this set of tuples is independent of $\alpha$, so we see that
the fibers are uniform in size and hence $|\mathcal{E}(F,G)|$ is the
product of $|\Sigma(G)|$ and the number of elements of $X(G,\alpha)^g$
generating $G$.

On the other hand, if we fix $\alpha \in \Sigma(G)$ then it is easily
seen that $G = Y(G,\alpha) X(G,\alpha)$ and that $X(G,\alpha) \cap
Y(G,\alpha) = \{1\}$. Associate to $\phi \in \mathcal{E}(F,G)$, the
tuple $(\phi(x_1),...,\phi(x_g)) \in G^g$. Write this uniquely as
$(a_1b_1,...,a_gb_g)$ where $a_i \in Y(G,\alpha)$ and $b_i \in
X(G,\alpha)$.  Since $\phi$ is surjective if and only if $b_1,...,b_g$
generate $G$ (as $Y(G,\alpha) \subseteq \Phi(G)$), we see that
$|\mathcal{E}(F,G)|$ is $|Y(G,\alpha)|^g$ times the number of elements
of $X(G,\alpha)^g$ generating $G$.

Equating the two expressions for $|\mathcal{E}(F,G)|$, we deduce that
$|\Sigma(G)| = |Y(G,\alpha)|^g = \see(G)^g$ as desired.
\end{proof}

Combining Theorem~\ref{main-group-theory} and
Theorem~\ref{aut-group-orders}, and using the function $\eta_j(p)$ in
Remark~\ref{remark-eta}, we obtain the following Corollary, which is
the basis for Conjecture \ref{our-conjecture} stated in the Introduction.
\begin{corollary}
Let $G$ be a \SSA group of $p$-class $c$ and rank
$g$ satisfying KIP. Let $\r=\r(G)$ and $r = r(G)$.  Then
\begin{equation*}
\Meas_c(G) = 
\frac{1}{|\Aut_\sigma(G)|}
\,
p^{g^2} \left(\frac{ \eta_g(p)^2}{\eta_{g-\r}(p) }\right) \\
\end{equation*}
and for $c' > c$
\begin{equation*}
\Meas_{c'}(G) = 
\frac{1}{|\Aut_\sigma(G)|}
\,
p^{g^2} \left(\frac{ \eta_g(p)^2}{\eta_{g-r}(p) }\right).
\end{equation*}
\end{corollary}

\begin{remark}
We have largely set aside the case of cyclic $p$-groups in this section
because, being abelian, they are already covered by the original
Cohen-Lenstra heuristics.  Hence, one can compute $\Meas^{ab}(G)$ as a predictor for the value of the
frequency $\Freq(G)$.  However, one can also view a cyclic $p$-group $G$ as a Schur
$\sigma$-group. It easy to see that such a group satisfies KIP and we can therefore compute $\Meas(G)$ via our formula; when we do so, we obtain the same value for $\Freq(G)$ since the GI-automorphism $\sigma$ is just inversion and so $\Aut_\sigma(G)=\Aut(G)$.
\end{remark}

For use later and to illustrate the ideas so far, we now display a
tree showing the first few levels of \SSA groups $G$ which are
descendants of $G_1$. Each vertex corresponds to a group $G$ and is
labeled with the quantity $\mathrm{Meas}_c(G)$ where $c$ is the
$p$-class of $G$. Certain relationships exist between the labels as
explained in Theorem~\ref{thm-meas-relations}. Vertices with no
descendants (circled in the figure) correspond to Schur
$\sigma$-groups and in this case the label is also the value of
$\Meas(G)$. For vertices that are not terminal, the label is always
equal to the sum of the labels of the immediate descendants.

These labels were calculated using the formulas from this section and assuming KIP. 
We confirmed KIP computationally for each group in the figure with $p$-class at most $6$. 
Unfortunately it would appear to be prohibitively time-consuming to test KIP for all of the groups at the next level.

In constructing the tree, a new phenomenon appeared. Namely, in addition to $G_1$, $G_2$ and $G_3$, there are two other $2$-generated $3$-groups of $3$-class~$2$ that have a GI-automorphism, $\Z/3 \times
\Z/9$ and $\Z/9 \times \Z/9$.  These groups arise as quotients of
Schur $\sigma$-groups (indeed all the groups above are quotients of
$G_3$) but not as $G/P_2(G)$ for any Schur $\sigma$-group $G$ and so 
are not \SSA groups.

These groups even have the difference between their $p$-multiplicator
rank and nuclear rank equal to $2$ and so are hard to distinguish from
\SSA groups.  We refer to such groups as
pseudo-\SSA groups. These arise elsewhere, although
rarely since in the situations we are considering, all the children of
a group typically have the same order and so one cannot be a proper
quotient of another. As an example, $J_{22}$, introduced below in
Section 4, has two \SSA groups and two pseudo-Schur
$\sigma$-quotient groups as children, which are quotients of one of
the \SSA groups by subgroups of order $3$ fixed by
the GI-automorphism.

\newpage

\includegraphics[scale=0.62,angle=90]{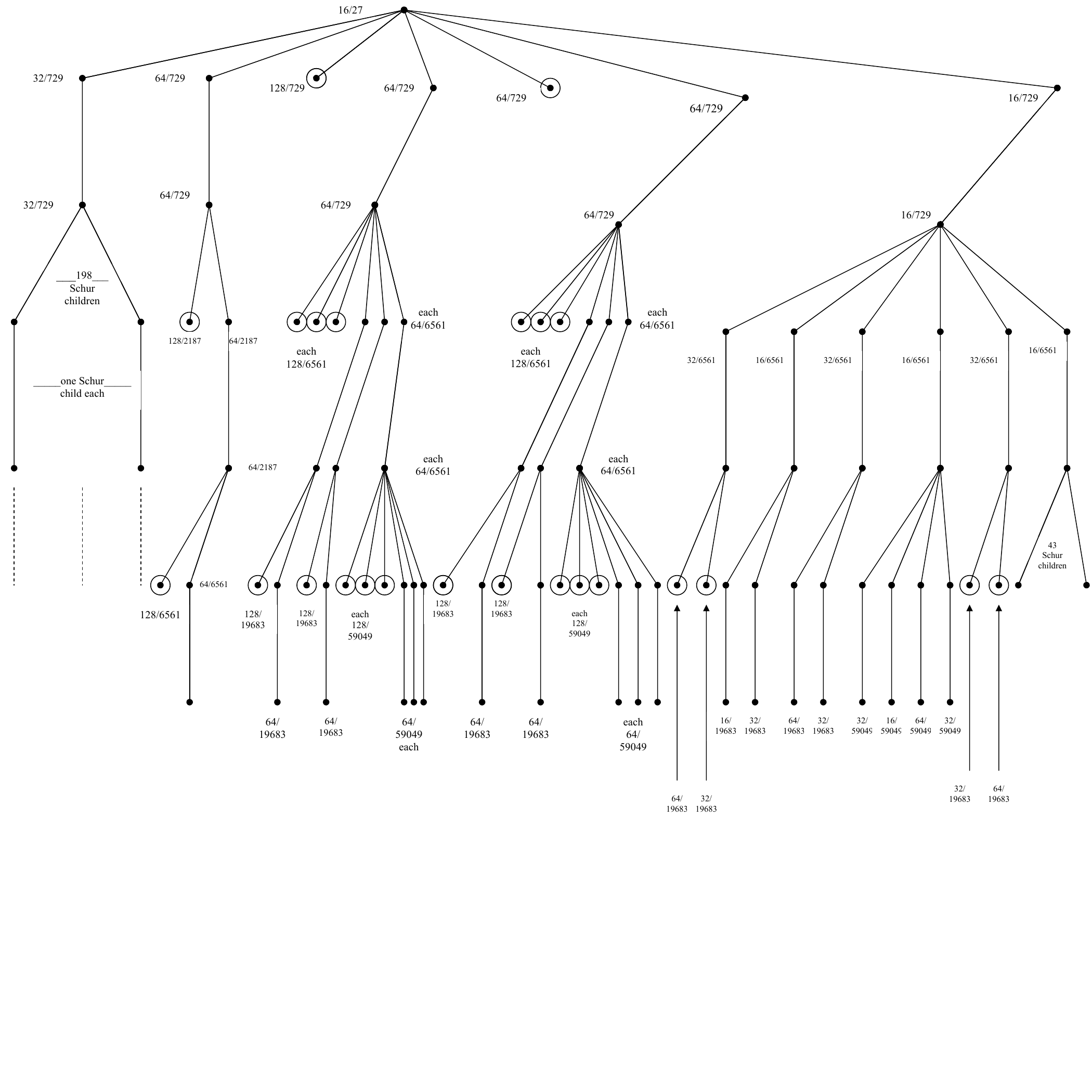}

\subsection{Groups satisfying KIP}\label{subsection-kip}

The KIP condition would seem to be quite
restrictive, yet it applies in all the cases where we have needed to
compute $\Meas_c(G)$ and we have yet to find a
\SSA group where it does not apply. Finite abelian $p$-groups certainly
satisfy KIP; we also have the following result.

\begin{theorem}\label{kernel-invariance}
For all $c \geq 1$, if $G = F_c$ or $G$ is an immediate descendant of $F_c$ possessing a GI-automorphism, then $G$ satisfies KIP.
\end{theorem}
\begin{proof}
If $G = F_c$ then every epimorphism $\alpha: F \rightarrow G$ factors
through the natural epimorphism $F \rightarrow F_c$ since $G$ has
$p$-class $c$.  This gives rise to an epimorphism from $F_c$ to $G$
which must be an isomorphism since $G = F_c$ is finite. It follows
that $\ker \alpha = P_c(F)$ which is a characteristic subgroup of $F$
and hence invariant under $\sigma$.

If $G$ is an immediate descendant of $F_c$ then every epimorphism
$\alpha: F \rightarrow G$ factors through the natural epimorphism $F
\rightarrow F_{c+1}$ and $P_{c+1}(F) \subseteq \ker \alpha \subseteq
P_c(F)$. Hence, all the kernels of epimorphisms from $F$ to $G$ will
be invariant under $\sigma$ if and only if all the kernels of the
epimorphisms from $F_{c+1}$ to $G$ are invariant.  As noted earlier,
if $G$ possesses a $\sigma$-automorphism then there must exist at
least one epimorphism with kernel that is invariant under $\sigma$. We
now show that this implies all such kernels are invariant.

Observe that if $R_1$ and $R_2$ are two such kernels then the
isomorphism $F/R_1 \cong F/R_2$ lifts to an automorphism of $F$ which
maps $R_1$ to $R_2$. It follows that $\Aut(F)$ acts transitively on
the set of kernels. If we now consider the images of the kernels in
$P_c(F)/P_{c+1}(F) \subseteq F_{c+1}$, then the same statement holds
where $\Aut(F)$ acts via the restriction homomorphism $\rho_c: \Aut(F)
\rightarrow \Aut(P_c(F)/P_{c+1}(F))$. The map $\rho_c$ factors through
$\rho_0: \Aut(F) \rightarrow \Aut(F/P_1(F))$. This can be seen by
using an inductive argument to verify that $\ker \rho_0 \subseteq \ker
\rho_i$ for $i \geq 0$. The induction is straightforward and uses the
recursive definition of the central series $\{ P_i(F) \}_{i \geq
  0}$. More details can be found in the proof of a slightly more
general statement appearing in \cite[Chapter VIII, Theorem 1.7]{HB}.

Since $\rho_0(\sigma)$ is the inversion automorphism in
$\Aut(F/P_1(F))$, it is clearly central in the image of $\Aut(F)$
under $\rho_0$. It follows that the image of $\sigma$ is central in
the image of $\Aut(F)$ under $\rho_i$ for all $i$. Combining this
statement for $i = c$ with the transitivity of the action of $\Aut(F)$
on the images of the kernels in $P_c(F)/P_{c+1}(F)$, we see that if
one image is invariant under $\sigma$ then they all must be
invariant. Pulling this back to $F$, we see that if one kernel is
invariant under $\sigma$ then they all must be invariant.
\end{proof}

There are examples of finite $p$-groups with GI-automorphisms that do not satisfy KIP.
The five groups SmallGroup(243,$i$) for $i = 51,...,55$ are the smallest ones. Indeed it appears that for any odd prime $p$ there are exactly five groups of order $p^5$ failing to satisfy KIP, all of which have $g=3$ generators
but $h=6$ and so are not \SSA groups. Among the groups of order $729$, there are exactly $58$ 
such examples, of which $53$ are $3$-generated (and have $h=5,6,$ or $8$) and five are 
$4$-generated (and have $h=10$). Therefore none of the examples of order $729$ are \SSA groups.

\section{Conjectures}

In this section, we formulate our main heuristic assumption, then use
the  group-theoretical results from the previous section to make precise
conjectures about the distribution of $p$-class tower groups  of
imaginary quadratic fields as well as the distribution of their maximal $p$-class $c$ quotients.
Recall that $A_K$ denotes the $p$-Sylow subgroup of the class group of $K$.

The arithmetic input, as already noted by Koch and Venkov, is
three-fold.  First, we observe that for an imaginary
quadratic field $K$, complex conjugation has a natural action on arithmetic objects
attached to $K$.  In particular, since $\Q$ has trivial class group,
$\a \overline{\a}$ is principal for every fractional ideal $\a$ of
$K$, so complex conjugation acts by inversion on $A_K$.  More
generally, complex conjugation acts as an involution on $\G_K$, and as
inversion on $\G_K^{\ab} \cong A_K$ thanks to the functorial
properties of the Artin reciprocity map.  The last two ingredients are
the finiteness of the class group, and the vanishing of the $p$-rank
of the unit group of $\mathcal{O}_K$.  The former ensures that $\G_K$
has finite abelianization (as does every one of its open subgroups),
and the latter that $r(\G_K)=d(\G_K)$, by a theorem of Shafarevich
\cite{Sh}.  Thus, $\G_K$ is always a Schur $\sigma$-group.

For $x>0$, let $\FF_x$ denote the set of imaginary quadratic fields
with absolute value of discriminant not exceeding $x$, and for each
natural number $g$, let $\FF_{x,g}$ be the subset of $\FF_x$
consisting of those fields $K$ having $d(A_K)=g$.  
For pro-$p$ groups $G$ and $H$, define  $\ch_G(H)$ to be $1$ if $G \cong H$ and $0$ otherwise.
\begin{definition}
Let $G$ be a  finitely generated pro-$p$ group with generator rank $g$.
We define 
$$ \Freq(G) = \lim_{x \to \infty} \frac{\sum_{K \in \FF_{x,g}}
\ch_G(\G_K)}{\sum_{K\in \FF_{x,g}} 1},
$$ 
assuming the limit exists.
If $G$ is also finite then, for $c \geq 1$, we define
$$ \Freq_c(G) = \lim_{x \to \infty} \frac{\sum_{K \in \FF_{x,g}}
\ch_G(\G_K/P_c(\G_K))}{\sum_{K\in \FF_{x,g}} 1},
$$ assuming the limit exists. 
\end{definition}

Our main heuristic assumption is that the frequencies defined above exist and are given by the group-theoretical measures 
introduced in Section~2 when $G$ is a finite $p$-group. More specifically, we make the following conjecture.

\begin{conjecture}\label{MainConjecture}
For every finite $p$-group $G$, we have 
\begin{eqnarray*}
 \Freq(G) &=&\Meas(G) \\
 \Freq_c(G) &=& \Meas_c(G).
\end{eqnarray*}
In particular, $\Freq(G) \neq 0$ if and only if $G$ is a Schur $\sigma$-group and $\Freq_c(G) \neq 0$ if and only if
$G$ is a \SSA group with $p$-class $c$ or $G$ is a Schur $\sigma$-group with $p$-class at most $c$.
When $G$ satisfies KIP, the measures can be computed using the formulas provided in Section~\ref{subsection-formulas}. 
\end{conjecture}

As a consequence of Conjecture \ref{MainConjecture}, we
expect every {\em finite} Schur $\sigma$-group (respectively \SSA group of $p$-class $c$) to occur as $\G_K$
(respectively $\G_K/P_c(\G_K)$) for a positive proportion of
imaginary quadratic fields $K$. 

We do not have a conjecture about the value of $\Freq(G)$ when $G$ is
an infinite pro-$p$ group. It is worth noting that there are
infinite Schur $\sigma$-groups that we do not expect to arise as
$\G_K$ for any $K$. For example, the Sylow $3$-subgroup of
$SL_2(\Z_3)$ considered in \cite{BB} is a $2$-generator $2$-relator
pro-$3$ group with finite abelianization and a
GI-automorphism, but the tame case of the Fontaine-Mazur conjecture~\cite[Conjecture 5a]{FM}
implies that it does not arise as $\G_K$ for any $K$. It is, however,
arbitrarily closely approximated by the finite Schur $\sigma$-groups
in \cite{BB}.

\section{Index-p-Abelianization-Data (IPAD)}  
\label{section-IPAD}

As discussed in the Introduction, a complete calculation of $\G_K$ is
prohibitive for most fields $K$. We thus seek to put certain partial but accessible information
about $p$-class tower groups under a general group-theoretical
framework.  In order to make comparisons with data coming from
number theory, it will be useful to consider abelianizations of low
index subgroups. To that end we introduce the notion of IPAD.  Thanks
to the $p$-group generation algorithm, and the theory developed in
Section 2, we are able to prove precise measures for the most frequent IPADs when $p = 3$ and $g = 2$. 
We will compare these values with  the observed number-theoretical frequencies in Section~5 (see
Table 2).

\begin{definition}
The abelian group $\Z/q_1 \times \dots \times \Z/q_d$ will be denoted
$[q_1,\dots,q_d]$.  Given a $g$-generated pro-$p$ group $G$, its
Index-$p$ Abelianization Data (or IPAD for short) will be the
unordered $ (p^g-1)/(p-1)$-tuple of abelianizations of the index $p$
subgroups of $G$ augmented by the abelianization of $G$ itself; we
always list the latter group first.
It
will be called $\ipad(G)$.
\end{definition}

For example, the IPAD of the Schur $\sigma$-group
{\tt SmallGroup}(243,5) will be denoted $$[ [3,3] ; [3,3,3] [3,9]^3 ],$$
indicating that its abelianization is $[3,3]$ and those of its $4$
index $3$ subgroups are $[3,3,3],$ $[3,9],$ $[3,9],$ and $ [3,9]$.  

Some other terminology that we will use in this section: for brevity, a 
descendant of a \SSA group $G$ is called a \emph{Schur descendant} of $G$ if it is also a \SSA group.
If it is an immediate descendant then we will call it a \emph{Schur child}. We will sometimes simply say that a group is
\emph{Schur} to indicate that it is a \SSA.

There are two things to note in working with IPADs \cite{BN}. First,
considering $g$-generated pro-$p$ groups for a fixed $p$ and $g$, if
$H$ is a quotient of $G$, then each entry of $\IPAD(H)$ is a quotient
of a corresponding entry of $\IPAD(G)$. This gives a partial order on
IPADs and we say that $\IPAD(H) \leq \IPAD(G)$. Second, if
$\IPAD(G/P_n(G)) = \IPAD(G/P_{n-1}(G))$ (we call the IPAD settled), then
$\IPAD(G) = \IPAD(G/P_n(G))$.

It follows that for a given IPAD there is a measurable subset of $X^g$
producing groups with that IPAD.  We now compute, in the case $g=2$
and $p=3$, the measures of the most common IPADs.

\begin{theorem}\label{group_theory_predictions}

(1) $\mathrm{IPAD} \ [ [3,3] ; [3,3,3]  [3,9]^3  ]$ has measure $128/729 \approx 0.1756$;

(2) $\mathrm{IPAD} \ [ [3,9] ; [3,3,9]^2 [3,27]^2 ]$ has measure $256/2187 \approx 0.1171$;
  
(3) $\mathrm{IPAD} \ [ [3,3] ; [3,3,3]^3, [3,9] ]$ has measure $64/729 \approx 0.0878$;

(4) $\mathrm{IPAD} \ [ [3,3] ; [3,3,3]^2 [3,9]^2 ]$ has measure $64/729 \approx 0.0878$;

(5) $\mathrm{IPAD} \ [ [3,3] ; [3,9]^3 [9,27] ]$ has measure $512/6561 \approx0.0780$;

(6) $\mathrm{IPAD} \ [ [3,3] ; [3,3,3] [3,9]^2 [9,27] ]$ has measure $512/6561 \approx0.0780$;

(7) $\mathrm{IPAD} \  [ [3,27] ; [3,3,27]^2 [3,81]^2 ]$ has measure $256/6561 \approx 0.0390$;

(8) $\mathrm{IPAD} \ [ [3,3] ; [3,3,3]^2 [9,27]^2  ]$ has measure $2048/59049 \approx 0.0347$;

(9) $\mathrm{IPAD} \  [ [3,9] ; [3,3,9] [3,9,27] [3,27]^2  ]$ has measure $640/19683 \approx 0.0325$;

(10) $\mathrm{IPAD} \  [ [3,3] ; [3,9]^4 ]$ has measure $16/729 \approx 0.0219$;

(11) $\mathrm{IPAD} \  [ [3,9] ; [3,3,9] [3,27]^3 ]$ has measure $128/6561 \approx 0.0195$;

(12) $\mathrm{IPAD} \  [ [3,9] ; [3,3,9] [3,27]^2 [9,9,9] ]$ has measure $128/6561 \approx 0.0195$;

(13)  $\mathrm{IPAD} \  [ [3,9] ; [3,3,3,3] [3,27]^3 ]$ has measure $128/6561 \approx 0.0195$;

(14) $\mathrm{IPAD} \  [ [3,9] ; [3,9,27] [3,27]^3 ]$ has measure $1024/59049 \approx 0.0173$.

\end{theorem}

\begin{proof} 
First, note that we checked that all the groups below whose measures were needed
in this computation, satisfy KIP. 

Next, note that the abelianizations of $G_1, G_2, G_3$ are $[3,3],
[3,9], [9,9]$ respectively.  It follows that any IPAD with first entry
$[3,3]$ has to come from descendants of $G_1$, and moreover that the first
entry is settled, and so every descendant of $G_1$ has abelianization
$[3,3]$.

Thus, for all the cases above starting with $[3,3]$, we focus on
descendants of $G_1$.  By O'Brien we compute that $G_1$ has $11$
children. Of these, $7$ have difference between $p$-multiplicator rank
and nuclear rank at most $2$ (in fact exactly $2$) and all of these
turn out to have a GI-automorphism. Call them $H_1, \dots, H_7$ in
the order produced by O'Brien's algorithm as implemented in \magma
  (Ver. 2.16).

Of these, $H_3$ and $H_5$ are terminal and so are Schur
$\sigma$-groups. In the standard database they are SmallGroup(243,5)
and SmallGroup(243,7) respectively. Their IPADs are those on lines (1)
and (4) above. We compute that $\Meas(H_3) = 128/729$ and $\Meas(H_5)
= 64/729$. (1) and (4) follow by establishing that none of the Schur
descendants of the other $H_i$ have these IPADs. This also shows that
these groups are determined by their IPADs.  Note that the latter fact for
SmallGroup(243,5) is already observed in \cite{BB}[Prop.~3.1 and Cor.~3.3].

Of the other IPADs, only $\mathrm{IPAD}(H_4) \leq \mathrm{IPAD}(H_3)$
(in fact equal). The Schur child of $H_4$ has IPAD including $[9,9]$
and so does not contribute to (1). As for (4), we need to consider
$H_1$, which has the same IPAD as $H_5$. Only one child of $H_1$,
however, is Schur and its IPAD includes a $[9,9]$ and so cannot
contribute to (4). Thus, (1) and (4) are complete.

The Schur child of $H_1$ has $1602$ children, of which $198$ are
Schur. All of these have IPAD $[ [3,3] ; [3,3,3]^2 [9,27]^2 ]$ 
and nuclear rank between $2$ and $4$. All the Schur children of
$155$ of these have the same IPAD, so are settled and they contribute
$2048/59049$ to line (8) above. The Schur children of the other $43$
include $[27,27]$, so do not count towards (8).  The IPADs of the
remaining $H_i$ are not less than or equal to this IPAD and so (8) is
also complete.

The IPAD of $H_2$ is that on line (3) and all its children have the
same IPAD. It therefore contributes Meas$(H_2) = 64/729$ to (3). None
of the other $H_i$ has small enough IPAD that their descendants could
have IPAD as in (3), and so (3) is proven.

The IPADs of $H_6$ and $H_7$ are both that given in line (10). All the
children of $H_6$ have IPADs involving $[9,9]$, whereas the IPADs of
all the children of $H_7$ are settled as (10). It follows that this
IPAD has measure Meas$(H_7) = 16/729$, proving (10).

As for cases (5) and (6), these come from further investigation of
descendants of $H_6$ and $H_4$ respectively. In each case, the group
has a unique Schur child, which then has $6$ Schur children.  These
all have the respective IPADs. In each case, $3$ of the $6$ are
terminal, and the other $3$ each have one Schur child. Two of those
are settled, whereas the remaining group has larger IPAD.  Thus $5$ of
the $6$ Schur grandchildren of each $H_i$, whose measures are each
$64/729$, contribute to (5) and (6) respectively and the remaining
grandchild, whose measure is $64/6561$, does not. Thus the IPADs in
(5) and (6) each have measure $64/729 - 64/6561 = 512/6561$, and (5)
and (6) are proven.

IPADs (2), (7), (9), (11), (12), (13), and (14) above must come from
descendants of $G_2$. This has $22$ Schur children.  We call these
$J_1, \dots, J_{22}$ in accordance with O'Brien's ordering. Only
$J_{10}, J_{11}$, and $J_{12}$ have IPADs less than or equal to (in
fact equal to) that of (2). The last two are terminal and the unique
Schur child of $J_{10}$ has larger IPAD. Thus, the IPAD of (2) has
measure $\Meas(J_{11}) +$ $\Meas(J_{12}) = 256/2187$, and (2) is proven.

The unique Schur child above has IPAD $[ [3,9] ; [3,3,9] [3,9,9]
  [3,27]^2 ]$. A Schur descendant of $G_2$ with IPAD in line (9) has
to descend from this child (by comparing the IPADs of the other
$J_i$). It has $9$ Schur children, of which $6$ have the IPAD of
(9). The others have IPAD $[ [3,9] ; [3,3,9] [9,9,9] [3,27]^2]$, which
is incomparable. Two of these are terminal, the other settled, and so
this proves line (12).  Of the remaining $6$, there are $4$ terminal
groups, $1$ settled, and $1$ with a unique Schur child with larger
IPAD. Summing the measures of the first $5$ groups yields $640/19683$
and establishes (9).

Case (7) can only arise from descendants of $J_5$. It has $3$ Schur
children, with the $2$ terminal ones having the desired IPAD and the
other having larger IPAD. This establishes (7).

Case (11) arises from descendants of $J_{14}$ and $J_{17}$, all of
which are settled, and so its measure is the sum of their
measures. Case (13) similarly arises from $J_{13}$ and $J_{16}$, which
are settled.

As for (14), this has to come from descendants of $J_{15}$ and
$J_{18}$. Each has measure $64/6561$ and their trees of descendants
are identical. Each has a unique Schur child and $4$ Schur
grandchildren. Of these, $1$ is terminal and $2$ others settled with
the desired IPAD.  The children of the remaining group have larger
IPAD, so we subtract its measure, $64/59049$.  Since $2(64/6561 -
64/59049) = 1024/59049$, (14) is proven.

Note that none of the $14$ given IPADs have first entry greater than or 
equal to $[9,9]$ and so no descendants of $G_3$ will have one of these IPADs.
Since the measure of $G_3$ is $1/81 = 0.0123$, the IPADs produced by
its descendants will all have measure smaller than that of any of the $14$ given IPADs.
\end{proof}

The descendants of $H_1, H_4,$ and $H_6$ appear to follow periodic
patterns that lead to the following conjecture, which would complete
the computation of measures of IPADs beginning $[3,3]$, since summing
all their conjectured values gives $16/27 = \mathrm{Meas}(G_1)$.

\begin{conjecture}
(a) If $k \geq 2$, then $\mathrm{IPAD} \ [ [3,3] ; [3,9]^3
    [3^k,3^{k+1}] ]$ has measure $512/3^{2k+4}$;

(b) If $k \geq 2$, then $\mathrm{IPAD} \ [ [3,3] ; [3,3,3] [3,9]^2
    [3^k,3^{k+1}] ]$ has measure $512/3^{2k+4}$;

(c) If $k \geq 2$, then $\mathrm{IPAD} \ [ [3,3] ; [3,3,3]^2
    [3^k,3^{k+1}]^2 ]$ has measure $2048/3^{4k+2}$;

(d) If $k \geq 2$, then $\mathrm{IPAD} \ [ [3,3] ; [3,3,3]^2
    [3^k,3^{k+1}] [3^{k+1},3^{k+2}] ]$ has measure $512/3^{4k+2}$.

\end{conjecture}

\begin{remark}
1. As noted, the measure of an IPAD is the sum of the measures of
terminal and settled groups.  If it only involves terminal groups,
then it determines a finite list of groups having that IPAD.
Sometimes, such as for lines (1) and (4) above, it determines a unique
group.  Now consider the IPAD in line (7), which corresponds to the
two terminal Schur children of $J_5$. An imaginary quadratic number
field with that IPAD (such as $\Q(\sqrt{-17399})$) therefore has one
of these two groups as the Galois group of its $3$-class tower, the
first cases of a non-abelian $3$-class tower of a quadratic field having $3$-class
length $4$. This group has derived length $2$. We have not found an
IPAD consisting only of terminal groups of finite derived length
exceeding $2$.

2. In \cite{KV}, Koch and Venkov proved that if a $2$-generated Schur
$\sigma$-group is finite, then it has relations at depth $3$ and $k$
where $k \in \{3,5,7\}$ in the $p$-Zassenhaus filtration.  McLeman
\cite{McL} conjectures that the group is finite if and only if both
relations have depth $3$.  Computing dimensions of the first three
factors of the Jennings series, we observe that every Schur descendant
of $G_1$ has its relations at this depth. The apparent combinatorial
explosion in descendants of $H_1$ then casts doubt on the ``if'' part
of McLeman's conjecture.

As for Schur descendants of $G_2$, those not having both relations at
depth $3$ are precisely those descended from $J_6, \dots, J_9, J_{19},
\dots, J_{22}$. The combinatorial explosion in descendants of these
groups lends support to the ``only if'' part of McLeman's conjecture.

3.  One might ask for the probability that a $2$-generated Schur
$\sigma$-$3$-group is finite. Searching through the tree, we find $90$
descendants of $G_1$ that are Schur $\sigma$-groups of $3$-class at
most $11$, $144$ descendants of $G_2$ that are Schur $\sigma$-groups
of $3$-class at most $8$, and $222$ descendants of $G_3$ that are
Schur $\sigma$-groups of $3$-class at most $7$. Their combined measure
is slightly over $0.8533$ and so, in this sense, there is at least an
$85.33\%$ probability that a $2$-generated Schur $\sigma$-$3$-group is
finite.

As for an upper bound, it is natural, in the spirit of Golod and
Shafarevich, to conjecture that ``large'' IPADs will correspond only
to infinite groups, but one must be careful.  Extending the above
census slightly, we find that $J_1$ has Schur $\sigma$-group
descendants of $3$-class $9$ and order $3^{18}$ with IPAD $[ [3,243] ;
  [3,3,3,81],[3,729]^3 ]$. Thus, having a rank $4$ subgroup of index
$3$ (the highest rank possible by comparison with the free group) is
not sufficient to imply that the Schur $\sigma$-group is infinite.

\end{remark}

\section{Computations}

As evidence for our conjectures we have collected numerical data in
the case of the smallest odd prime $p = 3$. In particular, we have obtained IPADs for  all imaginary 
quadratic fields $K$ with $3$-class group of rank $2$ and discriminant $d_K$
satisfying $|d_K| < 10^8$ assuming GRH. 

For an imaginary quadratic field $K$ of discriminant $-d$, with $3$-class group of rank 2, the four unramified cubic extensions $F$ of $K$ can be computed using results of Fung and Williams \cite{FungWilliams}. The maximal real subfield $F^+$ of $K$ is a cubic field of discriminant $-d$. Fung and Williams allows one to compute a defining polynomial for such fields efficiently.  Indeed, one finds many such polynomials, and we can then distinguish the four isomorphism classes of fields by using GP/Pari's program \verb?nfisisom?.  It is then straightforward to compute a defining polynomial for $F=F^+(\sqrt{-d})$ and compute its class group to obtain the desired IPAD. 

Originally, we had attempted to compute the unramified extensions of $K$ and their class groups directly but there was a large amount of variation in the running times and this approach proved to be very slow. The current approach was suggested in~\cite{Ma}. There, the author refers to an object called the Transfer Target Type (TTT) of $K$. The notion of TTT is almost the same as our IPAD except that the $3$-class group of the base field is not included.

Computations were carried out using the symbolic algebra package
\pari~\cite{pari} running on $2 \times 2.66$ GHz
6-Core Intel Xeon processors running OS X 10.8.5. 
The computations were run in parallel across multiple cores by dividing up the discriminants into subintervals and searching through a space of potential defining polynomials for the cubic extensions using the coefficient bounds in \cite{FungWilliams}. Although this created some redundancy, the parallelization limited the real world running time to the maximum length across all of the intervals. Roughly 890 core hours were used in total.

{\footnotesize
\begin{table}[htdp]
\caption{Census of the most common IPADs.}
\begin{center}

\begin{tabular}{|l |c|c|c|c|c|}
\hline
 & $I_1$ & $I_{3.2}$ & $I_{10}$ & $I_{32}$ & $I_{100}$\\
\hline
$[3,3]$; $[3,3,3]$ $[3,9]^3$   & 667 & 2270 & 7622 & 25737 & 83352 \\
\hline
$[3,9]$; $[3,3,9]^2$ $[3,27]^2$   & 406 & 1497 & 4974 & 16821 & 55308 \\
\hline
$[3,3]$; $[3,3,3]^2$ $[3,9]^2$   & 269 & 1069 & 3625 & 12314 & 41398 \\
\hline
$[3,3]$; $[3,3,3]^3$ $[3,9]$   & 297 & 1056 & 3619 & 12324 & 40967 \\
\hline
$[3,3]$; $[3,9]^3$ $[9,27]$   & 276 & 973 & 3190 & 11042 & 36457 \\
\hline
$[3,3]$; $[3,3,3]$ $[3,9]^2$ $[9,27]$   & 249 & 889 & 3113 & 10739 & 35922 \\
\hline
$[3,27]$; $[3,3,27]^2$ $[3,81]^2$   & 103 & 463 & 1615 & 5620 & 18422 \\
\hline
$[3,3]$; $[3,3,3]^2$ $[9,27]^2$   & 112 & 384 & 1293 & 4593 & 15540 \\
\hline
$[3,9]$; $[3,3,9]$ $[3,9,27]$ $[3,27]^2$   & 101 & 367 & 1317 & 4559 & 15037 \\
\hline
$[3,3]$; $[3,9]^4$  & 94 & 323 & 1019 & 3284 & 10426 \\
\hline
$[3,9]$; $[3,3,9]$ $[3,27]^3$   & 75 & 254 & 844 & 2914 & 9335 \\
\hline
$[3,9]$; $[3,3,3,3]$ $[3,27]^3$   & 64 & 233 & 799 & 2734 & 9000 \\
\hline
$[3,9]$; $[3,3,9]$ $[3,27]^2$ $[9,9,9]$   & 66 & 229 & 786 & 2740 & 8953 \\
\hline
$[3,9]$; $[3,9,27]$ $[3,27]^3$   & 61 & 232 & 728 & 2447 & 8165 \\
\hline
Other IPADs (341 types)  & 350 & 1505 & 5741 & 21222 & 73643 \\
\hline
\hline
Total   & 3190 & 11744 & 40285 & 139090 & 461925 \\
\hline
\end{tabular}
\end{center}
\end{table}
}
{\footnotesize
\begin{table}[htdp]
\caption{Relative proportions of the most common IPADs.}
\label{last-table}
\begin{center}

\begin{tabular}{|l |c|c|c|c|c|c|}
\hline
 & $I_1$ & $I_{3.2}$ & $I_{10}$ & $I_{32}$ & $I_{100}$& Predicted \\
\hline
$[3,3]$; $[3,3,3]$ $[3,9]^3$   & 0.2091 & 0.1933 & 0.1892 & 0.1850 & 0.1804 & 
0.1756 \\
\hline
$[3,9]$; $[3,3,9]^2$ $[3,27]^2$   & 0.1273 & 0.1275 & 0.1235 & 0.1209 & 0.1197 &
0.1171 \\
\hline
$[3,3]$; $[3,3,3]^2$ $[3,9]^2$   & 0.0843 & 0.0910 & 0.0900 & 0.0885 & 0.0896 & 
0.0878 \\
\hline
$[3,3]$; $[3,3,3]^3$ $[3,9]$   & 0.0931 & 0.0899 & 0.0898 & 0.0886 & 0.0887 & 
0.0878 \\
\hline
$[3,3]$; $[3,9]^3$ $[9,27]$   & 0.0865 & 0.0829 & 0.0792 & 0.0794 & 0.0789 & 
0.0780 \\
\hline
$[3,3]$; $[3,3,3]$ $[3,9]^2$ $[9,27]$   & 0.0781 & 0.0757 & 0.0773 & 0.0772 & 
0.0778 & 0.0780 \\
\hline
$[3,27]$; $[3,3,27]^2$ $[3,81]^2$   & 0.0323 & 0.0394 & 0.0401 & 0.0404 & 0.0399
& 0.0390 \\
\hline
$[3,3]$; $[3,3,3]^2$ $[9,27]^2$   & 0.0351 & 0.0327 & 0.0321 & 0.0330 & 0.0336 &
0.0347 \\
\hline
$[3,9]$; $[3,3,9]$ $[3,9,27]$ $[3,27]^2$   & 0.0317 & 0.0313 & 0.0327 & 0.0328 &
0.0326 & 0.0325 \\
\hline
$[3,3]$; $[3,9]^4$  & 0.0295 & 0.0275 & 0.0253 & 0.0236 & 0.0226 & 0.0219 \\
\hline
$[3,9]$; $[3,3,9]$ $[3,27]^3$   & 0.0235 & 0.0216 & 0.0210 & 0.0210 & 0.0202 & 
0.0195 \\
\hline
$[3,9]$; $[3,3,3,3]$ $[3,27]^3$   & 0.0201 & 0.0198 & 0.0198 & 0.0197 & 0.0195 &
0.0195 \\
\hline
$[3,9]$; $[3,3,9]$ $[3,27]^2$ $[9,9,9]$   & 0.0207 & 0.0195 & 0.0195 & 0.0197 & 
0.0194 & 0.0195 \\
\hline
$[3,9]$; $[3,9,27]$ $[3,27]^3$   & 0.0191 & 0.0198 & 0.0181 & 0.0176 & 0.0177 & 
0.0173 \\
\hline
Other IPADs (341 types)  & 0.1097 & 0.1282 & 0.1425 & 0.1526 & 0.1594 & 0.1717\\
\hline
\end{tabular}
\end{center}
\end{table}
}


We now present a summary of the data collected. The first table is a census
 of the most common IPADs. The second lists
their relative proportions obtained by dividing through by the total
number of fields examined. In addition, the last column of the second
table lists the values predicted by our conjectures as computed in
Theorem~\ref{group_theory_predictions}.  Note that in lines 1 and 4 of
Table 2, the IPAD determines the isomporphism type of the group, namely
SmallGroup(243,5) and SmallGroup(243,7) respectively.  Thus, on these two lines, the predicted and computed frequencies of these two groups can be compared, providing a direct test of our non-abelian heuristics.

We have broken down the
interval of discriminants $d_K$ with $1 < |d_K| <  10^8$ into $5$ nested subintervals $I_j$ 
where $I_j = \{ d_K \mid  1 \leq -d_K \leq j \cdot 10^6  \}$ and we have
selected values of $j$ so that the length of each successive subinterval is scaled by a factor of
$\sqrt{10} \approx 3.2$.


\

\noindent{\sc Acknowledgements.}  We acknowledge useful correspondence 
and conversations with Bettina Eick, Jordan Ellenberg, John Labute, Daniel Mayer, Cam McLeman, 
Eamonn O'Brien, and Melanie Matchett~Wood.  We are grateful to Jonathan Blackhurst 
for providing the Appendix.  We would also like to thank Joann Boston for 
drawing the figure in Section 2.

\

\begin{center}
\begin{Huge}
{Appendix. On the nucleus of certain $p$-groups}
\end{Huge}

\

\begin{Large}
{Jonathan Blackhurst}
\end{Large}
\end{center}

\

In this appendix we prove the proposition that if the Schur multiplier
of a finite non-cyclic $p$-group $G$ is trivial, then the nucleus of
$G$ is trivial. Our proof of the proposition will use the facts that a
$p$-group has trivial nucleus if and only if it has no immediate
descendants and that a finite group has trivial Schur multiplier if
and only if it has no non-trivial stem extensions, so we will begin by
recalling a few definitions.  For the definition of the lower
$p$-central series and $p$-class of a group, we refer to section 2 of
the article.

\noindent {\sc Definition.}  Let $G$ be a finite $p$-group with
minimal number of generators $d=d(G)$ and presentation $F/R$ where $F$
is the free pro-$p$ group on $d$ generators.  The \textit{$p$-covering
  group} $G^*$ of $G$ is $F/R^*$ where $R^*$ is the topological
closure of $R^p[F,R]$, and the \textit{nucleus} of $G$ is $P_c(G^*)$
where $c$ is the $p$-class of $G$.  The \textit{$p$-multiplicator} of
$G$ is defined to be the subgroup $R/R^*$ of $G^*$.  The \textit{Schur
  multiplier} $\mathcal{M}(G)$ of $G$ is defined to be $(R\cap
[F,F])/[F,R]$.  A group $C$ is a \textit{stem extension} of $G$ if
there is an exact sequence $$1\rightarrow K\rightarrow C\rightarrow
G\rightarrow 1$$ where $K$ is contained in the intersection of the
center and derived subgroups of $C$.

\

We will need to recall some basic properties of Schur multipliers and
$p$-covering groups. First, for a finite group $G$, the largest stem
extension of $G$ has size $|G||\mathcal{M}(G)|$.  Hence, the Schur
multiplier of a finite group $G$ is trivial if and only if $G$ admits
no non-trivial stem extensions.  Second, every elementary abelian
central extension of $G$ is a quotient of $G^*$. By this we mean that
if $H$ is a $d$-generated $p$-group with elementary abelian subgroup
$Z$ contained in the center of $H$ such that $H/Z$ is isomorphic to
$G$, then $H$ is a quotient of $G^*$.  Every immediate descendant of
$G$ is an elementary abelian central extension of $G$, hence is a
quotient of $G^*$.  A subgroup $M$ of the $p$-multiplicator of $G$ is
said to supplement the nucleus if $M$ and the nucleus together
generate the $p$-multiplicator, that is $MP_c(G^*)=R/R^*$.  The
immediate descendants of $G$ can be put in one-to-one correspondence
with equivalence classes of proper subgroups $M$ of the
$p$-multiplicator of $G$ that supplement the nucleus.  The equivalence
relation comes from the action of the outer automorphism group of
$G^*$, so $M$ and $N$ are equivalent if there is an outer automorphism
$\sigma$ of $G^*$ such that $\sigma(M)=N$. The reader is referred to
O'Brien \cite{O} for more details.

With these preliminaries in place, we can show that the non-cyclic hypothesis in our proposition is necessary by considering the finite cyclic $p$-group $G=\mathbb{Z}/p^c\mathbb{Z}$. The Schur multiplier is trivial since in this case $F=\mathbb{Z}$ so $[F,F]$ is trivial. On the other hand, the nucleus is non-trivial since in this case $F=\mathbb{Z}_p$ and $R=p^c\mathbb{Z}_p$ so $R^*=p^{c+1}\mathbb{Z}_p$ and $G^*=F/R^*=\mathbb{Z}/p^{c+1}\mathbb{Z}$ which implies that $P_c(G^*)=p^cG^*$ is non-trivial.

\

\noindent {\sc Proposition:} \textit{Let $G$ be a finite non-cyclic $p$-group. If the Schur multiplier of $G$ is trivial, then the nucleus of $G$ is trivial.}

\begin{proof} We will prove the following equivalent assertion: if the nucleus of $G$ is non-trivial, then $G$ has a non-trivial stem extension. We divide the problem into two cases depending on whether the abelianization of $G$ has stabilized; that is, whether the abelianization of an immediate descendant of $G$ can have larger order than the abelianization $G^{ab}$ of $G$. We will see that this is equivalent to whether or not $G^{ab}\simeq (G/P_{c-1}(G))^{ab}$ where $G$ has $p$-class $c$.

\noindent CASE 1: Suppose that $G^{ab}\simeq (G/P_{c-1}(G))^{ab}$ and that the nucleus of $G$ is non-trivial. Since the nucleus is non-trivial, $G$ has an immediate descendant $C$ and we have the following diagram $$1\rightarrow K\rightarrow C\rightarrow G\rightarrow 1$$ where $K=P_c(C)$. Note that since $C/P_{k}(C)\simeq G/P_{k}(G)$ for $k\leq c$, we have that $(C/P_{c-1}(C))^{ab}\simeq (C/K)^{ab}$. If $P_{c-1}(C)$ were not contained within the derived subgroup $C'$ of $C$, then its image $\overline{P_{c-1}}(C)$ in $C/C'$ would be non-trivial. Since $K=P_{c-1}(C)^p[C,P_{c-1}(C)]$, the image $\overline{K}$ of $K$ would be $\overline{P_{c-1}}(C)^p$ and thus would be stricly smaller than $\overline{P_{c-1}}(C)$. Now $(C/H)^{ab}\simeq (C/C')/\overline{H}$ for any $H\triangleleft C$, so, replacing $H$ with $K$ and $P_{c-1}(C)$, we see that $(C/P_{c-1}(C))^{ab}$ would be smaller than $(C/K)^{ab}$, contradicting that they are isomorphic. Thus $P_{c-1}(C)<C'$, hence $K<C'$, so $C$ is a stem
  extension of $G$. Since $G$ has a non-trivial stem extension, its Schur multiplier is non-trivial.

\noindent CASE 2: Suppose that $G^{ab} \not\simeq (G/P_{c-1}(G))^{ab}$. Let $$1\rightarrow R\rightarrow F\rightarrow G\rightarrow 1$$ be a presentation of $G$ where $F$ is free pro-$p$ group on $d$ generators and $d$ is the minimal number of generators of $G$. Induction and the argument in the preceding case shows that $(G/P_k(G))^{ab}$ is strictly smaller than $(G/P_{k+1}(G))^{ab}$ for any $k<c$. Furthermore, since the image $\overline{P_{k+1}}(G)$ of $P_{k+1}(G)$ in $G/G'$ is $\overline{P_k}(G)^p$, there must be a generator $b$ of $F$ such that the image of $b^{p^{c-1}}$ in $G$ lies outside $G'$. Now consider $R^*=R^p[F,R]$ and let $G^*=F/R^*$ be the $p$-covering group of $G$. We have the following diagrams: $$1\rightarrow R^*\rightarrow F\rightarrow G^*\rightarrow 1$$ and $$1\rightarrow R/R^*\rightarrow G^*\rightarrow G\rightarrow 1$$ 

We now show that the image of $b^{p^c}$ in $P_c(G^*)$ is non-trivial so $G$ has non-trivial nucleus. Let $G$ have abelianization isomorphic to $\mathbb{Z}/p^{n_1}\mathbb{Z}\times\cdots\times\mathbb{Z}/p^{n_d}\mathbb{Z}$. Consider the topological closure $S$ of $R\cup [F,F]$. Then $F/S$ is isomorphic to $G^{ab}$. The group $\mathbb{Z}/p^{n_1+1}\mathbb{Z}\times\cdots \times\mathbb{Z}/p^{n_d+1} \mathbb{Z}$ is an elementary abelian central extension of $F/S$. This implies that $b^{p^c}$ lies outside $S^*=S^p[F,S]$. Since $R\subset S$, we have that $R^*\subset S^*$. Hence $b^{p^c}$ lies outside $R^*$ so it has non-trivial image in $G^*$. Since its image lies inside $P_c(G^*)$, this group is non-trivial.

We have shown that $G$ has non-trivial nucleus. Now let $a$ be a generator of $F$ independent of $b$---i.e., one that doesn't map to the same element as $b$ in the elementary abelianization of $F$---and let $\overline{M}$ be a proper subgroup of $R/R^*$ that contains the image of $b^{p^c}[a,b^{p^{c-1}}]$ and that supplements the subgroup of $R/R^*$ generated by the image of $b^{p^c}$ (so $\overline{M}$ and the image of $b^{p^c}$ generate $R/R^*$). Now consider $C=G^*/\overline{M}$. Letting $K=(R/R^*)/\overline{M}$, we have the following diagram $$1\rightarrow K\rightarrow C\rightarrow G\rightarrow 1$$ Since $G^*$ is a central extension of $G$ and $C$ is a quotient of $G^*$, $C$ is also a central extension of $G$. Furthermore, $|K|=p$. Now let $M$ be the subgroup of $F$ corresponding to $\overline{M}$ under the lattice isomorphism theorem. Then we have the following diagram: $$1\rightarrow M\rightarrow F\rightarrow C\rightarrow 1$$ Since $M$ does not contain $b^{p^c}$, its image in $C
 $ is non-trivial. Since $G$ has $p$-class $c$, the image of $b^{p^c}$ is trivial in $G$. Also since $|K|=p$, the image of the powers of $b^{p^c}$ constitute $K$. Since $M$ does contain $b^{p^c}[a,b^{p^{c-1}}]$, the image of $b^{p^c}$ in $C$ equals the image of $[b^{p^{c-1}},a]$, hence $K$ lies in the derived subgroup of $C$, so $C$ is a non-trivial stem extension of $G$.  Consequently, the Schur multiplier of $G$ is non-trivial. 
 \end{proof}

\noindent Jonathan Blackhurst ~~~ \textsf{blackhur@math.wisc.edu}
\\
\noindent Department of Mathematics, University of Wisconsin - Madison,  
480 Lincoln Drive, Madison, WI 53706, USA

\end{document}